\documentclass[review,onefignum,onetabnum]{siamart190516}

\usepackage{lipsum}
\usepackage{amsfonts}
\usepackage{graphicx}
\usepackage{epstopdf}
\usepackage{algorithmic}
\usepackage{amsmath}
\usepackage{amssymb}
\usepackage{bm}
\usepackage{float}
\usepackage{multirow}
\usepackage{color}
\usepackage{booktabs}
\usepackage{multirow}
\usepackage{graphicx}
\usepackage{subfigure}
\usepackage{diagbox}
\usepackage{url}
\usepackage{verbatim}
\ifpdf
  \DeclareGraphicsExtensions{.eps,.pdf,.png,.jpg}
\else
  \DeclareGraphicsExtensions{.eps}
\fi

\newsiamremark{remark}{Remark}
\newsiamremark{hypothesis}{Hypothesis}
\crefname{hypothesis}{Hypothesis}{Hypotheses}
\newsiamthm{claim}{Claim}
\newsiamthm{example}{Example}
\newsiamthm{fact}{Fact}

\headers{Convergence Rate Analysis of AFBA with GN Momentum}{Y. Lin, S. Li and Y. Zhang}

\title{Convergence Rate Analysis of Accelerated Forward-Backward Algorithm with Generalized Nesterov Momentum Scheme\thanks{Submitted to the editors \today.}
\funding{This work is supported in part by Fundamental Research Funds for the Central Universities of China under Grant 11620352, by the Opening Project of Guangdong Province Key Laboratory of Computational Science at the Sun Yat-sen University under Grants 2021006 and 2021007, by Natural Science Foundation of China under Grants 11771464 and 11601537, and by the Science and Technology Program of Guangzhou under grant 201804020053.}}

\author{Yizun Lin\thanks{Department of Mathematics, College of Information Science and Technology, Jinan University, Guangzhou 510632, China (\email{linyizun@jnu.edu.cn}).}
\and Si Li\thanks{Corresponding author. School of Computer Science and Technology, Guangdong University of Technology, Guangzhou 510006, China (\email{sili@gdut.edu.cn}).}
\and Yunzhong Zhang\thanks{Jinan University-University of Birmingham Joint Institute, Jinan University, Guangzhou 511443, China (\email{zhangyunzhong1999@163.com}).}
}

\usepackage{amsopn}
\DeclareMathOperator*{\argmin}{argmin}

\def \bR {\mathbb R}
\def \bN {\mathbb N}

\def \mI {\mathcal{I}}
\def \mK {\mathcal{K}}
\def \mL {\mathcal{L}}

\def \mT {\mathcal{T}}
\def \mN {\mathcal{N}}

\def \bRn {\bR^{n}}
\def \bRm {\bR^{m}}

\def \sign {\mathrm{sign}}
\def \prox {\mathrm{prox}}
\def \diag {\mathrm{diag}}

\ifpdf
\hypersetup{
  pdftitle={An Example Article},
  pdfauthor={D. Doe, P. T. Frank, and J. E. Smith}
  frenchlinks={false}
}
\fi

\begin{document}

\maketitle

% REQUIRED
\begin{abstract}
Nesterov's accelerated forward-backward algorithm (AFBA) is an efficient algorithm for solving a class of two-term convex optimization models consisting of a differentiable function with a Lipschitz continuous gradient plus a nondifferentiable function with a closed form of its proximity operator. It has been shown that the iterative sequence generated by AFBA with a modified Nesterov's momentum scheme converges to a minimizer of the objective function with an $o\left(\frac{1}{k^2}\right)$ convergence rate in terms of the function value (FV-convergence rate) and an $o\left(\frac{1}{k}\right)$ convergence rate in terms of the distance between consecutive iterates (DCI-convergence rate). In this paper, we propose a more general momentum scheme with an introduced power parameter $\omega\in(0,1]$ and show that AFBA with the proposed momentum scheme converges to a minimizer of the objective function with an $o\left(\frac{1}{k^{2\omega}}\right)$
FV-convergence rate and an $o\left(\frac{1}{k^{\omega}}\right)$ DCI-convergence rate. The generality of the proposed momentum scheme provides us a variety of parameter selections for different scenarios, which makes the resulting algorithm more flexible to achieve better performance. We then employ AFBA with the proposed momentum scheme to solve the smoothed hinge loss $\ell_1$-support vector machine model. Numerical results demonstrate that the proposed generalized momentum scheme outperforms two existing momentum schemes.

\end{abstract}
%
% REQUIRED

\section{Introduction}
In this paper, we consider fast algorithm with a generalized Nesterov momentum scheme for solving a class of two-term optimization problems of the form
\begin{equation}\label{model1}
\min_{x\in\bRn}\{f(x)+g(x)\},
\end{equation}
where $f:\bRn\to\bR$ is a convex and differetiable function with a Lipschitz continuous gradient, $g:\bRn\to\bR$ is a proper lower-semicontinuous convex function which may not be differentiable. This two-term optimization model has important applications in machine learning (e.g. LASSO regression, support vector machine) \cite{tibshirani1996regression, zhang2001text, zhu20041}, image processing (e.g. image denoising, image restoration) \cite{daubechies2004iterative, elad2006image, figueiredo2003algorithm}, compressed sensing \cite{figueiredo2007gradient, yin2008bregman} and so on.

The possible nondifferentiability of $g$ in model \eqref{model1} precludes the use of classical gradient type algorithms. Under these circumstances, the forward-backward algorithm (FBA) \cite{lions1979splitting, parikh2014proximal} was developed to solve the model when the proximity operator of $g$ has a closed-form. The FBA is easily-implemented and robust. However, for large scale ill-conditioned problems, it has been shown to be too slow no matter in practice or in the sense of asymptotic rate of convergence \cite{beck2009fast,bredies2008linear}. To address this issue, various modifications of FBA have been developed \cite{beck2009fast,bertsekas2011incremental,combettes2016stochastic}. One of the most popular strategies is the utilization of momentum technique, such as Nesterov's momentum scheme \cite{nesterov1983method}. Beck and Teboulle showed that FBA has an $O\left(\frac{1}{k}\right)$ convergence rate in terms of the function value (FV-convergence rate), and FBA with Nesterov's momentum (fast iterative shrinkage-thresholding algorithm, FISTA) can improve the FV-convergence rate to $O(\frac{1}{k^2})$. However, the convergence of the iterative sequence generated by FISTA is unclear in their work \cite{beck2009fast}. Chambolle and Dossal proved in \cite{chambolle2015convergence} not only the $O(\frac{1}{k^2})$ FV-convergence rate but also the convergence of the iterative sequence for the momentum accelerated forward-backward algorithm with a new setting of momentum parameters (AFBA-CD). Later, Attouch and Peypouquet showed that AFBA-CD can actually achieve an $o\left(\frac{1}{k^2}\right)$ FV-convergence rate and an $o\left(\frac{1}{k}\right)$ convergence rate in terms of the distance between consecutive iterates (DCI-convergence rate) \cite{attouch2016rate}. Although AFBA-CD is theoretically guaranteed to be faster than FISTA, it does not always give a distinguishingly improved performance on practical applications.

In this work, we propose a more general setting of momentum parameters in the accelerated forward-backward algorithm (AFBA). A power parameter $\omega\in(0,1]$ is introduced in our momentum scheme. We shall show that the setting of momentum parameters in \cite{chambolle2015convergence} is a special case of the proposed generalized scheme with $\omega=1$. More importantly, the iterative sequence generated by AFBA with the generalized momentum scheme converges to a minimizer of the objective function with an $o\left(\frac{1}{k^{2\omega}}\right)$ FV-convergence rate and an $o\left(\frac{1}{k^\omega}\right)$ DCI-convergence rate. This result provides a wider class of momentum algorithms with various convergence rates. Numerical results demonstrate that the proposed momentum scheme outperforms the existing momentum schemes used in \cite{beck2009fast} and \cite{chambolle2015convergence} for classification problems using support vector machine (SVM).

We organize this paper in six sections. In section \ref{sec_AFBA}, we describe the accelerated forward backward algorithm and three types of momentum schemes, including two existing schemes and the proposed generalized scheme. We analyze in section \ref{mainsec} the convergence of the iterative sequence and both the FV-convergence rate and the DCI-convergence rate for AFBA with the proposed momentum scheme. In section \ref{sec_AFBAforSVM}, we formulate the smoothed hinge loss $\ell_1$-SVM model as the two-term optimization model \eqref{model1}, and then employ AFBA to solve this model. Section \ref{sec_numerical} presents the numerical results for comparison of the proposed momentum scheme with the other two schemes mentioned in section \ref{sec_AFBA}. Section \ref{sec_conclusion} offers a conclusion.

\section{Accelerated forward-backward algorithm}\label{sec_AFBA}
In this section, we first review the accelerated forward-backward algorithm (AFBA) for solving Model \eqref{model1} and two existing momentum schemes. Inspired by these two schemes, we then propose a more general setting of momentum parameters. To better describe the iteration scheme of AFBA, we recall the definition of proximity operator of a convex function \cite{moreau1965proximite}. For $x,y\in\bRn$, the inner product is defined by $\langle x,y\rangle:=\sum_{i=1}^{n}x_iy_i$, and the corresponding $\ell_2$ norm is given by $\|x\|:=\langle x,x\rangle^\frac{1}{2}$.
\begin{definition}
Let $\psi:\bRn\to\bR\cup\{+\infty\}$ be a proper convex function. The proximity operator of $\psi$ at $x\in\bRn$ is defined by
\begin{equation}
\prox_\psi(x):=\argmin_{u\in\bRn}\left\{\frac{1}{2}\|u-x\|^2+\psi(u)\right\}.
\end{equation}
\end{definition}

%a proposition characterizes the relationship between the proximity operator and the subdifferential of a convex function \cite{micchelli2011proximity, moreau1965proximite}.

Throughout this paper, we define $F:=f+g$ and
\begin{equation}\label{def_operT}
T:=\prox_{\beta g}\circ(\mI-\beta\nabla f),
\end{equation}
where $\mI$ denotes the identity operator. We will always assume that the minimizer of $F$ exists and let $\beta\in\left(0,\frac{1}{L}\right]$. We say that an operator $\mT:\bRn\to\bRn$ is nonexpansive if $\|\mT x-\mT y\|\leq\|x-y\|$ for all $x,y\in\bRn$. If there exist $\alpha\in(0,1)$ and a nonexpansive operator $\mN:\bRn\to\bRn$ such that $\mT=(1-\alpha)\mI+\alpha\mN$, then we say that $\mT$ is $\alpha$-averaged nonexpansive. According to the proof of Theorem 26.14 in \cite{bauschke2017convex}, we know that operator $T$ in \eqref{def_operT} is averaged nonexpansive, which implies that the sequence generated by the fixed-point iteration $x^{k+1}=Tx^{k}$ (forward-backward iteration) converges to a fixed point of $T$ (see Proposition 5.16 in \cite{bauschke2017convex}). In addition, by employing Fermat's rule (Theorem 16.3 of \cite{bauschke2017convex}) and Proposition 2.6 of \cite{micchelli2011proximity}, we have the following equivalence between the fixed point of $T$ and the minimizer of $F$.
\begin{proposition}[\cite{parikh2014proximal}]\label{thm_minequiFP}
Vector $x^*\in\bRn$ is a minimizer of $F$ if and only if $x^*$ is a fixed point of $T$.
\end{proposition}

To sum up, the sequence generated by the fixed-point iteration of $T$ converges to a minimizer of $F$. Based on this fixed-point iteration, an accelerated version with momentum technique (AFBA) can be written as follows:
\begin{equation}\label{alg1}
\begin{cases}
y^k=x^k+\theta_k(x^k-x^{k-1}),\\
x^{k+1}=Ty^k.
\end{cases}
\end{equation}
To proceed the above iteration, two initial vectors $x^0,x^1\in\bRn$ should be given.

Next, we review two existing momentum schemes. We denote the set of all nonnegative integers and the set of all positive integers by $\bN_0$ and $\bN_+$, respectively. Let $\bR_+$ denote the set of all positive real numbers. For two sequences $\{a_k\}_{k\in\bN_0}\subset\bR_+\cup\{0\}$ and $\{b_k\}_{k\in\bN_0}\subset\bR_+$, both tending to zero, if $\lim_{k\to\infty}\frac{a_k}{b_k}=0$, we write $a_k=o(b_k)$. If there exist constants $c>0$ and $K\in\bN_0$ such that $a_k\leq cb_k$ for all $k>K$, we write $a_k=O(b_k)$.

The most popular way of setting the momentum parameters $\{\theta_k\}_{k\in\bN_+}$ is given by Nesterov \cite{nesterov1983method} as follows:
\begin{equation}\label{Nesterovthetak}
\theta_k=\frac{t_{k-1}-1}{t_k}\ \ \mbox{with}\ \ t_k=\frac{1+\sqrt{1+4t_{k-1}^2}}{2},\ \ k\in\bN_+,
\end{equation}
where $t_0=1$. When $\{\theta_k\}_{k\in\bN_+}$ in Algorithm \eqref{alg1} is set to Nesterov's momentum scheme \eqref{Nesterovthetak}, the algorithm reduces to the well-known fast iterative shrinkage-thresholding algorithm (FISTA) \cite{beck2009fast}. It has been shown that the convergence rate in terms of the function value (FV-convergence rate) of FISTA is $O\left(\frac{1}{k^2}\right)$. Later, Chambolle and Dossal \cite{chambolle2015convergence} proved the convergence of the iterative sequence generated by AFBA with the following setting of momentum parameters
\begin{equation}\label{CDthetak}
\theta_k=\frac{k-1}{k+\alpha-1},\ \ \alpha>3,\ k\in\bN_+.
\end{equation}
We note that the setting \eqref{Nesterovthetak} of $\{\theta_k\}_{k\in\bN_+}$ is asymptotically equivalent to \eqref{CDthetak} with $\alpha=3$ by Proposition 2 of \cite{lin2019krasnoselskii}. In the recent work \cite{attouch2016rate}, Attouch and Peypouquet proved that AFBA with momentum scheme \eqref{CDthetak} can achieve an $o\left(\frac{1}{k^2}\right)$ FV-convergence rate and an $o\left(\frac{1}{k}\right)$ convergence rate in terms of the distance between consecutive iterates (DCI-convergence rate).

In this paper, we investigate the convergence, FV-convergence rate and DCI-convergence rate of AFBA with a more general setting of $\theta_k$:
\begin{equation}\label{ourthetak}
\theta_k=\frac{t_{k-1}-1}{t_k}\ \ \mbox{with} \ \ t_{k-1}=a(k-1)^\omega+b,\ \ k\in\bN_+,
\end{equation}
where $\omega\in(0,1]$ and $a\in\bR_+$. To avoid division by zero, without further mentioning, we always set $b\in\bR\backslash\{-ak^\omega:k\in\bN_+\}$ throughout the paper. It is easy to see that the setting \eqref{CDthetak} is a special case of \eqref{ourthetak} with $\omega=1$, $a=\frac{1}{\alpha-1}$ and $b=1$. In subsequent section, we shall show that both the FV-convergence rate and the DCI-convergence rate of AFBA with the generalized momentum scheme \eqref{ourthetak} depend on the order of $t_{k-1}$.

\section{Convergence and convergence rate analysis}\label{mainsec}
In this section, we always let $\{x^k\}_{k\in\bN_0}$ and $\{y^k\}_{k\in\bN_+}$ be two sequences generated by Algorithm \eqref{alg1} for any two initial vectors $x^0,x^1\in\bRn$, and let $x^*$ be any fixed point of $T$, that is, any minimizer of the objective function $F$. We shall show that if the sequence of momentum parameters $\{\theta_k\}_{k\in\bN_+}$ is given by \eqref{ourthetak}, the sequence $\{x^k\}_{k\in\bN_0}$ converges to a minimizer of $F$ with an $o\left(\frac{1}{k^{2\omega}}\right)$ FV-convergence rate and an $o\left(\frac{1}{k^\omega}\right)$ DCI-convergence rate. We begin with stating our main theorem of this section.
\begin{theorem}\label{mainthm}
Suppose that $\{\theta_k\}_{k\in\bN_+}$ is given by \eqref{ourthetak}. If either $\omega\in(0,1)$, $a\in\bR_+$  or $\omega=1$, $a\in\left(0,\frac{1}{2}\right)$ holds, then we have the following facts:
\begin{itemize}
\item[$(i)$] $\|x^{k}-x^{k-1}\|=o\left(\frac{1}{k^\omega}\right)$,
\item[$(ii)$] $F(x^k)-F(x^*)=o\left(\frac{1}{k^{2\omega}}\right)$,
\item[$(iii)$] $\{x^k\}_{k\in\bN_0}$ converges to a minimizer of $F$.
\end{itemize}
\end{theorem}

We postpone the proof of Theorem \ref{mainthm} until we finish the establishment and verification of {\it Momentum-Condition}, which is sufficient to ensure the convergence and the desired convergence rate of AFBA. We first recall Lemma 2.3 of \cite{beck2009fast}.
\begin{lemma}\label{lemma_FzFx}
Let $x$, $y$ be any two vectors in $\bRn$ and set $z:=Ty$. Then
$$
F(z)\leq F(x)+\frac{1}{\beta}\langle y-x,y-z\rangle-\frac{1}{2\beta}\|y-z\|^2.
$$
\end{lemma}

For notational simplicity, throughout this section, we let
\begin{equation}\label{def_etaktauk}
\eta_k:=F(x^k)-F(x^*),\ \ \tau_k:=\frac{1}{2\beta}\|x^{k}-x^{k-1}\|^2,\ \ k\in\bN_+,
\end{equation}
and define the sequence $\{z^k\}_{k\in\bN_+}\subset\bRn$ by
\begin{equation}\label{def_zk}
z^k:=t_ky^k+(1-t_k)x^k,\ \ k\in\bN_+.
\end{equation}

By employing Lemma \ref{lemma_FzFx}, we next establish the following proposition that serves as an important tool in the analysis of convergence and convergence rate.
\begin{proposition}\label{prop_facts}
Let $\theta_k=\frac{t_{k-1}-1}{t_{k}}$, where $t_k\neq0$ for all $k\in\bN_+$, and define
\begin{equation}\label{def_epsilonk}
\varepsilon_k:=2\beta t_{k-1}^2\eta_k+\|z^k-x^*\|^2,\ \ k\in\bN_+.
\end{equation}
If there exists $K\in\bN_+$ such that $t_k(t_k-1)\leq t_{k-1}^2$ for all $k>K$, then the following facts hold:
\begin{itemize}
\item[$(i)$] $\varepsilon_{k+1}\leq\varepsilon_{k}$ for all $k>K$ and $\lim_{k\to\infty}\varepsilon_k$ exists,
\item[$(ii)$] $\eta_k\leq\frac{\varepsilon_K}{2\beta t_{k-1}^2}$ for all $k>K$,
\item[$(iii)$] $\sum_{k=1}^{\infty}\left[t_{k-1}^2-t_k(t_k-1)\right]\eta_k\leq\frac{\varepsilon_1}{2\beta}$.
\end{itemize}
\end{proposition}
\begin{proof}
We first prove Fact $(i)$. For $k\in\bN_+$, by letting $x=x^k$, $y=y^k$ and $x=x^*$, $y=y^k$, respectively, in Lemma \ref{lemma_FzFx}, we have that
\begin{align}
\label{ineq_FxkkFxk}F(x^{k+1})\leq F(x^k)+\frac{1}{2\beta}\left(2\langle y^k-x^{k},y^k-x^{k+1}\rangle-\|y^k-x^{k+1}\|^2\right),\\
\label{ineq_FxkkFxstar}F(x^{k+1})\leq F(x^*)+\frac{1}{2\beta}\left(2\langle y^k-x^*,y^k-x^{k+1}\rangle-\|y^k-x^{k+1}\|^2\right).
\end{align}
Let $p_k:=2t_k\langle z^k-x^{*},y^k-x^{k+1}\rangle-t_k^2\|y^k-x^{k+1}\|^2$, $k\in\bN_+$. By noting that
\begin{align*}
\left(1-\frac{1}{t_k}\right)(y^k-x^k)+\frac{1}{t_k}(y^k-x^*)=\frac{1}{t_k}\left[t_ky^k+(1-t_k)x^k-x^*\right]=\frac{1}{t_k}(z^k-x^*),
\end{align*}
the combination $\left(1-\frac{1}{t_k}\right)\cdot\eqref{ineq_FxkkFxk}$+$\frac{1}{t_k}\cdot\eqref{ineq_FxkkFxstar}$ gives that
$$
F(x^{k+1})\leq\left(1-\frac{1}{t_k}\right)F(x^k)+\frac{1}{t_k}F(x^*)+\frac{1}{2\beta t_k^2}p_k,
$$
that is,
\begin{equation}\label{ineq_FxkkFxkFxstar} \eta_{k+1}\leq\left(1-\frac{1}{t_k}\right)\eta_{k}+\frac{1}{2\beta t_k^2}p_k,\ \ \mbox{for all}\ \ k\in\bN_+.
\end{equation}
To prove Fact $(i)$, we also need to verify the following equality
\begin{equation}\label{eq_zkp1zk}
z^{k+1}=z^k-t_k(y^k-x^{k+1}),\ \ \mbox{for all}\ \ k\in\bN_+.
\end{equation}
Substituting $y^{k+1}=x^{k+1}+\theta_{k+1}(x^{k+1}-x^{k})$ into the definition of $z^{k+1}$ in \eqref{def_zk}, and then using the facts $t_{k+1}\theta_{k+1}=t_k-1$ and $(1-t_k)x^k=z^k-t_ky^k$, we get that
\begin{align}
\notag z^{k+1}&=t_{k+1}\left[x^{k+1}+\theta_{k+1}(x^{k+1}-x^{k})\right]+(1-t_{k+1})x^{k+1}\\
\notag&=(1+t_{k+1}\theta_{k+1})x^{k+1}-t_{k+1}\theta_{k+1}x^{k}\\
\label{eq_zkp1tkxk}&=t_kx^{k+1}+(1-t_k)x^k\\
\notag&=t_kx^{k+1}+z^k-t_ky^k,
\end{align}
which implies \eqref{eq_zkp1zk}. Since $p_k=\|z^k-x^*\|^2-\|(z^k-x^*)-t_k(y^k-x^{k+1})\|^2$, it follows from \eqref{eq_zkp1zk} that
\begin{equation}\label{eq_zkxstar}
p_k=\|z^k-x^*\|^2-\|z^{k+1}-x^*\|^2.
\end{equation}
Substituting \eqref{eq_zkxstar} into \eqref{ineq_FxkkFxkFxstar} yields that
\begin{equation}\label{ineq_FxkkFxkFxstar2} \eta_{k+1}\leq\left(1-\frac{1}{t_k}\right)\eta_{k}+\frac{1}{2\beta t_k^2}\left(\|z^k-x^*\|^2-\|z^{k+1}-x^*\|^2\right).
\end{equation}
Multiplying both sides of \eqref{ineq_FxkkFxkFxstar2} by $2\beta t_k^2$ gives that
\begin{align*}
2\beta t_k^2\eta_{k+1}&\leq2\beta t_k(t_k-1)\eta_{k}+\|z^k-x^*\|^2-\|z^{k+1}-x^*\|^2\\
&=2\beta t_{k-1}^2\eta_{k}+\|z^k-x^*\|^2-\|z^{k+1}-x^*\|^2-2\beta\left[t_{k-1}^2-t_k(t_k-1)\right]\eta_{k},
\end{align*}
that is,
\begin{equation}\label{ineq_epsilon}
\varepsilon_{k+1}+2\beta\left[t_{k-1}^2-t_k(t_k-1)\right]\eta_{k}\leq\varepsilon_k,\ \ \mbox{for all}\ \ k\in\bN_+.
\end{equation}
Since $\eta_{k}\geq0$ and there exists $K\in\bN_+$ such that $t_k(t_k-1)\leq t_{k-1}^2$ for all $k>K$, Fact $(i)$ follows from \eqref{ineq_epsilon} immediately.

According to the definition of $\varepsilon_{k}$ and Fact $(i)$, we have that
$$
2\beta t_{k-1}^2\eta_{k}\leq\varepsilon_k\leq\varepsilon_K\ \ \mbox{for all}\ \ k>K,
$$
which implies Fact $(ii)$. Summing \eqref{ineq_epsilon} for $k=1,\ldots,K$ and using the fact $\varepsilon_{K+1}\geq0$, we obtain that
$$
\sum_{k=1}^{K}2\beta\left[t_{k-1}^2-t_k(t_k-1)\right]\eta_{k}\leq\varepsilon_1-\varepsilon_{K+1}\leq\varepsilon_1,
$$
which proves Fact $(iii)$.
\end{proof}

As a direct result of Fact $(ii)$ in Proposition \ref{prop_facts}, the following corollary can recover the $O\left(\frac{1}{k^2}\right)$ FV-convergence rate of FISTA shown in \cite{beck2009fast}.
\begin{corollary}\label{cor_FISTA}
Let $\theta_k=\frac{t_{k-1}-1}{t_{k}}$, where $t_k\neq0$ for all $k\in\bN_+$. If $t_k>0$ and $t_k(t_k-1)=t_{k-1}^2$ for all $k\in\bN_+$, then $F(x^k)-F(x^*)=O(\frac{1}{k^2})$.
\end{corollary}
\begin{proof}
Solving the quadratic equation $t_k(t_k-1)=t_{k-1}^2$ with unknown $t_k$, we obtain that $t_k=\frac{1\pm\sqrt{1+4t_{k-1}^2}}{2}$. Since $t_k>0$, it is necessary to choose
\begin{equation}\label{Nesterovtk}
t_k=\frac{1+\sqrt{1+4t_{k-1}^2}}{2},\ \ k\in\bN_+.
\end{equation}
According to \eqref{Nesterovtk}, we can verify by mathematical induction that $t_k>\frac{k+1}{2}$ holds for all $k\in\bN_+$, which together with Fact $(ii)$ in Proposition \ref{prop_facts} implies that $F(x^k)-F(x^*)=O\left(\frac{1}{k^2}\right)$.
\end{proof}

To obtain the convergence and convergence rate results of AFBA with momentum setting \eqref{ourthetak}, we need some hypotheses on the momentum parameters, which shall be used frequently in the rest of this section. For a sequence $\{t_k\}_{k\in\bN_0}\subset\bR$, we say that it satisfies {\it Momentum-Condition} if the following hypotheses are satisfied:
\begin{itemize}
\item[$(i)$] $t_k\neq0$ for all $k\in\bN_+$.
\item[$(ii)$] There exist $\rho\in\bR_+$ and $K_1\in\bN_+$ such that
\begin{equation}\label{neq_momcondtk1}
1\leq t_{k-1}<\rho\left[t_{k-1}^2-t_k(t_k-1)\right],\ \ \mbox{for all}\ \ k>K_1.
\end{equation}
\item[$(iii)$] There exist $c_1,c_2\in\bR_+$ and $K_2\in\bN_+$ such that
\begin{equation}\label{neq_momcondtk2}
c_1t_{k}\leq t_{k-1}\leq c_2t_{k},\ \ \mbox{for all}\ \ k>K_2.
\end{equation}
\item[$(iv)$] $\lim_{k\to\infty}t_k=+\infty$ and $\sum_{k=1}^{\infty}\frac{1}{t_k}=+\infty$.
\end{itemize}

We now establish the boundness of two series $\sum_{k=1}^{\infty}t_{k-1}\eta_k$ and $\sum_{k=1}^{\infty}t_{k-1}\tau_k$, which are crucial for the proof of higher-order infinitesimal $o(\cdot)$ convergence rate. The boundness of the former series is a direct result of Fact $(iii)$ in Proposition \ref{prop_facts}.

\begin{proposition}\label{prop_tkobj}
Let $\theta_k=\frac{t_{k-1}-1}{t_{k}}$, $k\in\bN_+$, where $\{t_k\}_{k\in\bN_0}\subset\bR$ satisfies Item $(i)$ and $(ii)$ of Momentum-Condition.
Then $\sum_{k=1}^{\infty}t_{k-1}\eta_k<+\infty$.
\end{proposition}
\begin{proof}
Multiplying both sides of the second inequality of \eqref{neq_momcondtk1} by $\eta_k$ and summing the resulting inequality for $k$ from $K_1+1$ to infinity yields that
$$
\sum_{k=K_1+1}^{\infty}t_{k-1}\eta_k<\rho\sum_{k=K_1+1}^{\infty}\left[t_{k-1}^2-t_k(t_k-1)\right]\eta_k,
$$
which implies the desired result by using Fact $(iii)$ in Proposition \ref{prop_facts}.
\end{proof}

We next prove the boundness of the other series as follows.
\begin{proposition}\label{prop_tkloc2}
Let $\theta_k=\frac{t_{k-1}-1}{t_{k}}$, $k\in\bN_+$, where $\{t_k\}_{k\in\bN_0}\subset\bR$ satisfies Item $(i)$--$(iii)$ of Momentum-Condition. Then $\sum_{k=1}^{\infty}t_{k-1}\tau_k<+\infty$.
\end{proposition}
\begin{proof}
We first show that
\begin{equation}\label{etaktaukineq}
\eta_{k+1}+\tau_{k+1}\leq\eta_k+\theta_k^2\tau_k,\ \ \mbox{for all}\ \ k\in\bN_+.
\end{equation}
From the proof of Proposition \ref{prop_facts}, we know that \eqref{ineq_FxkkFxk} holds. Substituting $y^k=x^k+\theta_k(x^k-x^{k-1})$ into \eqref{ineq_FxkkFxk} yields that
\begin{align}
\notag F(x^{k+1})&\leq F(x^k)+\frac{1}{\beta}\langle \theta_k(x^k-x^{k-1}),(x^k-x^{k+1})+\theta_k(x^k-x^{k-1})\rangle\\
\notag&\hspace{12pt}-\frac{1}{2\beta}\|(x^k-x^{k+1})+\theta_k(x^k-x^{k-1})\|^2\\
\label{ineqfortktauk}&=F(x^k)+\frac{1}{2\beta}\theta_k^2\|x^k-x^{k-1}\|^2-\frac{1}{2\beta}\|x^{k+1}-x^k\|^2.
\end{align}
Subtracting $F(x^*)$ from both sides of \eqref{ineqfortktauk} and recalling the definitions of $\eta_k$ and $\tau_k$ in \eqref{def_etaktauk}, we obtain \eqref{etaktaukineq}.

Multiplying both sides of \eqref{etaktaukineq} by $t_k^2$ yields that
\begin{equation}\label{ineq1_tauketak}
t_k^2(\eta_{k+1}+\tau_{k+1})\leq t_k^2\eta_k+(t_{k-1}-1)^2\tau_k,
\end{equation}
that is,
\begin{equation}\label{ineq2_tauketak}
(2t_{k-1}-1)\tau_k+\left(t_k^2\tau_{k+1}-t_{k-1}^2\tau_k\right)\leq t_k^2(\eta_k-\eta_{k+1}),\ \ \mbox{for all}\ \ k\in\bN_+.
\end{equation}
It follows from Item $(ii)$ and $(iii)$ of Momentum-Condition that there exist $c\in\bR_+$ and $K\in\bN_+$ such that $t_{k-1}\geq1$, $0<t_{k+1}(t_{k+1}-1)<t_{k}^2$ and $0<t_{k+1}\leq ct_k$ for all $k>K$,
which together with \eqref{ineq2_tauketak} give that
\begin{align*}
t_{k-1}\tau_k+\left(t_k^2\tau_{k+1}-t_{k-1}^2\tau_k\right)&\leq t_{k}^2\eta_k-t_{k+1}(t_{k+1}-1)\eta_{k+1}\\
&\leq\left(t_{k}^2\eta_k-t_{k+1}^2\eta_{k+1}\right)+ct_{k}\eta_{k+1}, \ \ \mbox{for all}\ \ k>K.
\end{align*}
Summing the above inequality for $k=K+1,K+2,\ldots,M$, we obtain that
$$
\sum_{k=K+1}^{M}t_{k-1}\tau_k+t_M^2\tau_{M+1}-t_{K}^2\tau_{K+1}\leq t_{K+1}^2\eta_{K+1}-t_{M+1}^{2}\eta_{M+1}+c\sum_{k=K+1}^{M}t_k\eta_{k+1},
$$
which yields that
\begin{equation}\label{ineq1_fortaukrate}
\sum_{k=K+1}^{M}t_{k-1}\tau_k\leq t_{K}^2\tau_{K+1}+t_{K+1}^2\eta_{K+1}+c\sum_{k=K+1}^{M}t_k\eta_{k+1}.
\end{equation}
Now letting $M\to\infty$ in \eqref{ineq1_fortaukrate} and using Proposition \ref{prop_tkobj}, we find that
$$
\sum_{k=K+1}^{\infty}t_{k-1}\tau_k<+\infty,
$$
which implies the desired result.
\end{proof}

We now apply the boundness of the above two series to establish the convergence rate, which can be achieved via proving $\lim_{k\to\infty}t_{k-1}^2\left(\tau_k+\eta_k\right)=0$. For this purpose, we need the following technical lemma.
\begin{lemma}\label{lem_existlim}
Let $\{a_k\}_{k\in\bN_0}\subset\bR$ be a sequence with a lower bound, and $\{b_k\}_{k\in\bN_+}\subset\bR$ be a sequence satisfying $\sum_{k=1}^{\infty}b_k<+\infty$. If there exists $K\in\bN_+$ such that $b_k\geq0$ and $a_{k}-a_{k-1}\leq b_k$ hold for all $k>K$, then $\lim_{k\to\infty}a_k$ exists.
\end{lemma}
\begin{proof}
By the facts $a_{k}-a_{k-1}\leq b_k$ and $b_k\geq0$, we have that
$$
a_k-a_{K}=\sum_{j=K+1}^{k}(a_j-a_{j-1})\leq\sum_{j=K+1}^{k}b_j\leq\sum_{j=K+1}^{\infty}b_j,\ \ \mbox{for all}\ \ k>K,
$$
which together with $\sum_{k=1}^{\infty}b_k<+\infty$ implies that $\{a_k\}_{k\in\bN_0}$ has an upper bound. Since $\{a_k\}_{k\in\bN_0}$ also has a lower bound, we know that there exists a subsequence $\{a_{k_j}\}_{j\in\bN_+}$ of $\{a_{k}\}_{k\in\bN_0}$ converging to some $a^*\in\bR$. We next prove that $\{a_{k}\}_{k\in\bN_0}$ also converges to $a^*$.

Let $\varepsilon>0$ be arbitrary. Since $\lim_{j\to\infty}a_{k_j}=a^*$, there exists $J_1\in\bN_+$ such that $a^*-\varepsilon<a_{k_j}<a^*+\varepsilon$ for all $j\geq J_1$. In addition, we note that $\sum_{k=1}^{\infty}b_k<+\infty$ and $b_k\geq0$ for all $k>K$. There exists $J_2\in\bN_+$ such that $\sum_{i=k_{J_2}}^{\infty}b_i<\varepsilon$. Let $J=\max\{J_1,J_2\}$ and $k>k_J$. Since there exists $J'\in\bN_+$ such that $k_{J'}>k$, we have that
$$
a_k=a_{k_{J'}}-\sum_{i=k}^{k_{J'}-1}(a_{i+1}-a_i)\geq a_{k_{J'}}-\sum_{i=k+1}^{k_{J'}}b_i>a^*-2\varepsilon.
$$
In addition,
$$
a_k=a_{k_J}+\sum_{i=k_{J}+1}^{k}(a_i-a_{i-1})\leq a_{k_{J}}+\sum_{i=k_{J}+1}^{k}b_i<a^*+2\varepsilon.
$$
We conclude that for any $\varepsilon>0$, there exists $J\in\bN_+$ such that $|a_k-a^*|<2\varepsilon$ holds for all $k>k_{J}$, which implies that $\lim_{k\to\infty}a_k=a^*$.
\end{proof}

\begin{proposition}\label{prop_limlocobj}
Let $\theta_k=\frac{t_{k-1}-1}{t_{k}}$, $k\in\bN_+$, where $\{t_k\}_{k\in\bN_0}\subset\bR$ satisfies Momentum-Condition. Then $\lim_{k\to\infty}t_{k-1}^2\left(\tau_k+\eta_k\right)=0$.
\end{proposition}
\begin{proof}
To simplify the notation, we let $p_k:=\tau_k+\eta_k$, $k\in\bN_+$. We now prove the existence of $\lim_{k\to\infty}t_{k-1}^2p_k$ by employing Lemma \ref{lem_existlim} with $a_{k-1}:=t_{k-1}^2p_k$ and $b_k:=t_{k}\eta_{k}$, $k\in\bN_+$. It is obvious that $\{t_{k-1}^2p_k\}_{k\in\bN_+}$ has a lower bound. By Item $(ii)$ and $(iii)$ of Momentum-Condition, there exist $c\in\bR_+$ and $K_1\in\bN_+$ such that $t_k\leq ct_{k-1}$ and $t_{k-1}\geq1$ for all $k>K_1$, which together with Proposition \ref{prop_tkobj} give that
$$
\sum_{k=K_1+1}^{\infty}t_{k}\eta_{k}\leq\sum_{k=K_1+1}^{\infty}ct_{k-1}\eta_{k}<+\infty,
$$
that is, $\sum_{k=1}^{\infty}t_{k}\eta_{k}<+\infty$. It remains to be shown that there exists $K\in\bN_+$ such that $t_{k}\eta_{k}\geq0$ and
\begin{equation}\label{tktauketakneq1}
t_k^2p_{k+1}-t_{k-1}^2p_k\leq t_{k}\eta_{k},\ \ \mbox{for all}\ \ k>K.
\end{equation}

Using Item $(ii)$ of Momentum-Condition again, there exists $K>K_1$ such that $t_{k-1}^2>t_{k}(t_{k}-1)\geq0$ for all $k>K$. The nonnegativity of $t_{k}\eta_{k}$ for $k>K$ can be obtained by $t_{k}\geq1$ immediately. We notice from the proof of Proposition \ref{prop_tkloc2} that \eqref{ineq1_tauketak} holds for all $k\in\bN_+$. The inequality $t_{k-1}\geq1$ also gives that $(t_{k-1}-1)^2<t_{k-1}^2$ for all $k>K$, which together with \eqref{ineq1_tauketak} implies
$$
t_k^2(\eta_{k+1}+\tau_{k+1})\leq t_k^2\eta_k+t_{k-1}^2\tau_k,
$$
that is,
\begin{equation}\label{etatauineq2}
t_k^2\tau_{k+1}-t_{k-1}^2\tau_k\leq t_k^2(\eta_{k}-\eta_{k+1}),\ \ \mbox{for all}\ \ k>K.
\end{equation}
In addition, we obtain from the fact $t_{k-1}^2>t_{k}(t_{k}-1)>0$ that
\begin{equation}\label{tketakeq1}
t_{k}^2\eta_{k+1}-t_{k-1}^2\eta_{k}\leq t_k^2(\eta_{k+1}-\eta_{k})+t_k\eta_{k},\ \ \mbox{for all}\ \ k>K.
\end{equation}
Adding the two inequalities \eqref{etatauineq2} and \eqref{tketakeq1} yields \eqref{tktauketakneq1}. We have now completed the proof that $\lim_{k\to\infty}t_{k-1}^2p_k$ exists.

Next, we prove that $\lim_{k\to\infty}t_{k-1}^2p_k=0$ holds by contradiction. Suppose that $\lim_{k\to\infty}t_{k-1}^2p_k\neq0$. Then there must exist some $s>0$ such that $\lim_{k\to\infty}t_{k-1}^2p_k=s$, since $p_k\geq0$ for all $k\in\bN_+$. This implies that there exists $K_2>K_1$ such that $t_{k-1}^2p_k>\frac{s}{2}$ and $t_{k-1}\geq1$ for all $k>K_2$. As a result, we have that
$$
\sum_{k=K_2+1}^\infty t_{k-1}p_k=\sum_{k=K_2+1}^\infty\frac{1}{t_{k-1}}\cdot t_{k-1}^2p_k>\frac{s}{2}\sum_{k=K_2+1}^\infty\frac{1}{t_{k-1}},
$$
which tends to $+\infty$ by Item $(iv)$ of Momentum-Condition. However, it follows from Proposition \ref{prop_tkobj} and Proposition \ref{prop_tkloc2} that $\sum_{k=1}^\infty t_{k-1}p_k<+\infty$. We have thus reached a contradiction. This completes the proof.
\end{proof}

With Proposition \ref{prop_limlocobj}, we are able to establish the convergence rate. To further prove the convergence of the iterative sequence, we also need the following lemma.
\begin{lemma}\label{lem_2condconvFP}
Let $\mT:\bRn\to\bRn$ be a nonexpansive operator such that it has at least one fixed point. If sequence $\{v^k\}_{k\in\bN_0}\subset\bRn$ satisfies the following two conditions:
\begin{itemize}
\item[$(i)$] $\lim_{k\to\infty}\|\mT v^k-v^k\|=0$,
\item[$(ii)$] $\lim_{k\to\infty}\|v^k-v^*\|$ exists for any fixed point $v^*$ of $\mT$,
\end{itemize}
then $\{v^k\}_{k\in\bN_0}$ converges to a fixed point of $\mT$.
\end{lemma}
\begin{proof}
We know from Item $(ii)$ that $\{v^k\}_{k\in\bN_0}$ is bounded. Hence there exists a subsequence $\{v^{k_j}\}_{j\in\bN_+}$ of $\{v^k\}_{k\in\bN_0}$ converging to some $\hat{v}\in\bRn$. We next prove that $\hat{v}$ is a fixed point of $\mT$. By the nonexpansiveness of $\mT$, we have that
$$
\lim_{j\to\infty}\|\mT\hat{v}-\mT v^{k_j}\|\leq\lim_{j\to\infty}\|\hat{v}-v^{k_j}\|=0,
$$
which implies that $\mT\hat{v}=\lim_{j\to\infty}\mT v^{k_j}$. This together with Item $(i)$ implies that
$$
\mT\hat{v}-\hat{v}=\lim_{j\to\infty}(\mT v^{k_j}-v^{k_j})={\bm0},
$$
that is, $\hat{v}$ is a fixed point of $\mT$. Now using Item $(ii)$ again with $v^*=\hat{v}$, we conclude that
$$
\lim_{k\to\infty}\|v^k-\hat{v}\|=\lim_{j\to\infty}\|v^{k_j}-\hat{v}\|=0,
$$
which completes the proof.
\end{proof}

We are now in a position to prove a theorem that is more general than Theorem \ref{mainthm}. We shall show that both the FV-convergence rate and the DCI-convergence rate of the sequence generated by AFBA with $\theta_k=\frac{t_{k-1}-1}{t_k}$ depend on the order of $t_{k-1}$ when $\{t_k\}_{k\in\bN_0}$ satisfies Momentum-Condition.

\begin{theorem}\label{maingeneralthm}
Suppose that $\theta_k=\frac{t_{k-1}-1}{t_{k}}$, $k\in\bN_+$, where $\{t_k\}_{k\in\bN_0}\subset\bR$ satisfies Momentum-Condition. Then the following hold:
\begin{itemize}
\item[$(i)$] $\|x^{k}-x^{k-1}\|=o\left(\frac{1}{t_{k-1}}\right)$,
\item[$(ii)$] $F(x^k)-F(x^*)=o\left(\frac{1}{t_{k-1}^2}\right)$,
\item[$(iii)$] $\{x^k\}_{k\in\bN_+}$ converges to a minimizer of $F$.
\end{itemize}
\end{theorem}
\begin{proof}
We first prove Item $(i)$ and $(ii)$ together by employing Proposition \ref{prop_limlocobj}. Since $\{t_k\}_{k\in\bN_0}$ satisfies Momentum-Condition, we know from Proposition \ref{prop_limlocobj} that
$$
\lim_{k\to\infty}t_{k-1}^2\left(\tau_k+\eta_k\right)=0.
$$
Recalling the definitions of $\tau_k$ and $\eta_k$ in \eqref{def_etaktauk}, we see that
\begin{equation}\label{eq_limDCIFV0}
\lim_{k\to\infty}{t_{k-1}^2}\|x^k-x^{k-1}\|^2=0\ \ \mbox{and}\ \
\lim_{k\to\infty}{t_{k-1}^2}\left(F(x^k)-F(x^*)\right)=0.
\end{equation}
Then Item $(i)$ and $(ii)$ of this theorem follow from \eqref{eq_limDCIFV0} and the fact $\lim_{k\to\infty}t_k=+\infty$ in Momentum-Condition.

We next employ Lemma \ref{lem_2condconvFP} to prove Item $(iii)$. As mentioned in section \ref{sec_AFBA}, $T$ is averaged nonexpansive, and hence nonexpansive. According to Lemma \ref{lem_2condconvFP}, it suffices to show that $\lim_{k\to\infty}\|Tx^k-x^k\|=0$ and $\lim_{k\to\infty}\|x^k-x^*\|$ exists for any fixed point $x^*$ of $T$, which shall be presented as follows.

It is easy to see from Momentum-Condition that $\{\theta_k\}_{k\in\bN_+}$ is bounded. In addition, it follows from the first equality in \eqref{eq_limDCIFV0} that $\lim_{k\to\infty}\|x^{k}-x^{k-1}\|=0$. These together with the nonexpansiveness of $T$ yield that
\begin{align*}
\lim_{k\to\infty}\|Tx^k-x^{k+1}\|&=\lim_{k\to\infty}\|Tx^k-T(x^{k}+\theta_k(x^{k}-x^{k-1}))\|\\
&\leq\lim_{k\to\infty}|\theta_k|\|x^k-x^{k-1}\|=0.
\end{align*}
Hence
$$
\lim_{k\to\infty}\|Tx^k-x^k\|\leq\lim_{k\to\infty}\left(\|Tx^k-x^{k+1}\|+\|x^{k+1}-x^{k}\|\right)=0,
$$
which implies that $\lim_{k\to\infty}\|Tx^k-x^k\|=0$.

Let $x^*$ be any fixed point of $T$. It remains to be shown that $\lim_{k\to\infty}\|x^k-x^*\|$ exists. To this end, we prove the existence of $\lim_{k\to\infty}\|z^k-x^*\|$, where $\{z_k\}_{k\in\bN_+}$ is defined by \eqref{def_zk}. We know from Fact $(i)$ of Proposition \ref{prop_facts} and the second equality in \eqref{eq_limDCIFV0} that both $\lim_{k\to\infty}\varepsilon_k$ and $\lim_{k\to\infty}2\beta t_{k-1}^2\eta_k$ exist, where $\varepsilon_k$ is defined by \eqref{def_epsilonk}. Hence $\lim_{k\to\infty}\|z^k-x^*\|$ exists. We now prove the existence of $\lim_{k\to\infty}\|x^k-x^*\|$. From \eqref{eq_zkp1tkxk} in the proof of Proposition \ref{prop_facts}, we see that $z^{k+1}=t_{k}(x^{k+1}-x^{k})+x^{k}$. Letting $r_k:=|t_{k-1}|\|x^k-x^{k-1}\|$ for $k\in\bN_+$ and using the triangle inequality, we have
$$
\|x^{k}-x^*\|-r_{k+1}\leq\|z^{k+1}-x^*\|\leq \|x^{k}-x^*\|+r_{k+1},
$$
that is,
\begin{equation}\label{neq_zkdomxk}
\|z^{k+1}-x^*\|-r_{k+1}\leq\|x^{k}-x^*\|\leq\|z^{k+1}-x^*\|+r_{k+1},\ \ k\in\bN_+.
\end{equation}
We also see from the first equality in \eqref{eq_limDCIFV0} that $\lim_{k\to\infty}r_k=0$. Now, the inequalities in \eqref{neq_zkdomxk} together with the existence of $\lim_{k\to\infty}\|z^k-x^*\|$ and the fact $\lim_{k\to\infty}r_k=0$ imply that $\lim_{k\to\infty}\|x^k-x^*\|$ exists. Therefore, Item $(iii)$ of this theorem follows from Lemma \ref{lem_2condconvFP} and Proposition \ref{thm_minequiFP}.
\end{proof}

We next show that $\{t_k\}_{k\in\bN_0}$ given in \eqref{ourthetak} satisfies Momentum-Condition. For this purpose, we need the fact that
\begin{equation}\label{lem_kpowomega}
\lim_{k\to\infty}\frac{k^{\omega}-(k-1)^{\omega}}{k^{\omega-1}}=\omega,\ \ \omega\in\bR,
\end{equation}
which can be verified by using L'Hopital's Rule.

\begin{proposition}\label{prop_tkmomcond}
Let $t_k:=ak^\omega+b$, $k\in\bN_0$, where $\omega\in(0,1]$, $a\in\bR$ and $b\in\bR\backslash\left\{-ak^\omega:k\in\bN_+\right\}$. If either $\omega\in(0,1)$, $a\in\bR_+$ or $\omega=1$, $a\in\left(0,\frac{1}{2}\right)$ holds, then $\{t_k\}_{k\in\bN_0}$ satisfies Momentum-Condition.
\end{proposition}
\begin{proof}
It is obvious that Item $(i)$ and $(iv)$ of Momentum-Condition hold for any $\omega\in(0,1]$ and $a\in\bR_+$. Next, we prove that $\{t_k\}_{k\in\bN_0}$ satisfies Item $(iii)$.

It follows from \eqref{lem_kpowomega} that for any $\omega\in(0,1]$, there exists $K_1\in\bN_+$ such that
\begin{equation}\label{kpowomegaleq2}
0<k^{\omega}-(k-1)^{\omega}<2\ \ \mbox{for all}\ \ k>K_1.
\end{equation}
Let $c_1=\frac{1}{2}$, $c_2=1$. It is obvious that $t_{k-1}\leq c_2t_k$ for all $k\in\bN_+$. Setting $K_2=1+\left\lceil\left|2-\frac{b}{a}\right|^\frac{1}{\omega}\right\rceil$ and $K_3=\max\left\{K_1,K_2\right\}$, by \eqref{kpowomegaleq2}, we have that for all $k>K_3$,
\begin{align*}
2t_{k-1}-t_k&=t_{k-1}-a(k^\omega-(k-1)^\omega)\\
&>a(K_2-1)^\omega+b-2a\geq0,
\end{align*}
which implies that $t_{k-1}\geq c_1t_k$. As a result, for any $\omega\in(0,1]$ and $a\in\bR_+$, $c_1t_{k}\leq t_{k-1}\leq c_2t_{k}$ holds for all $k>K_3$.

It remains to be shown the validity of Item $(ii)$ of Momentum-Condition. For any $\omega\in(0,1]$ and $a\in\bR_+$, we let $K_4=1+\left\lceil\left|\frac{1-b}{a}\right|^\frac{1}{\omega}\right\rceil$. Then for all $k>K_4$,
$$
t_{k-1}\geq a(K_4-1)^\omega+b\geq|1-b|+b\geq1.
$$
To complete the proof, it suffices to show that there exist $\rho\in\bR_+$ and $K\geq K_4$ such that
$$
t_{k-1}<\rho\left[t_{k-1}^2-t_k(t_k-1)\right],
$$
that is,
\begin{equation}\label{tkm1lesrho}
\frac{t_k}{t_{k-1}}-\left(\frac{t_k}{t_{k-1}}+1\right)(t_k-t_{k-1})>\frac{1}{\rho},\ \ \mbox{for all}\ \ k>K.
\end{equation}
It has been shown that for any $\omega\in(0,1]$ and $a\in\bR_+$, $\frac{1}{2}t_{k}\leq t_{k-1}\leq t_{k}$, that is, $1\leq\frac{t_k}{t_{k-1}}\leq2$ holds for all $k>K_3$. If $0<w<1$, then for any $a\in\bR_+$, it follows from \eqref{lem_kpowomega} that
$$
\lim_{k\to\infty}(t_k-t_{k-1})=\lim_{k\to\infty}a(k^\omega-(k-1)^\omega)=0,
$$
which implies that there exists $K_5\in\bN_+$ such that $0<t_k-t_{k-1}<\frac{1}{4}$ for all $k>K_5$. Now by setting $\rho=4$ and $K=\max\{K_3,K_4,K_5\}$, inequality \eqref{tkm1lesrho} holds.

If $\omega=1$ and $a\in\left(0,\frac{1}{2}\right)$, then $t_{k}-t_{k-1}=a$. Let $\varepsilon=\frac{1-2a}{2-a}$. Then $a=\frac{1-2\varepsilon}{2-\varepsilon}$ and $\varepsilon\in\left(0,\frac{1}{2}\right)$. We note that $\lim_{k\to\infty}\frac{t_k}{t_{k-1}}=1$. Hence there exists $K_6\in\bN_+$ such that $\frac{t_k}{t_{k-1}}>1-\varepsilon$ for all $k>K_6$. Now by setting $\rho=\frac{1}{\varepsilon}$ and $K=\max\{K_4,K_6\}$, we have that for all $k>K$,
\begin{align*}
\frac{t_k}{t_{k-1}}-\left(\frac{t_k}{t_{k-1}}+1\right)(t_k-t_{k-1})&=\frac{t_k}{t_{k-1}}-a\left(\frac{t_k}{t_{k-1}}+1\right)\\
&>(1-a)(1-\varepsilon)-a\\
&=(1-\varepsilon)-(2-\varepsilon)a\\
&=\varepsilon=\frac{1}{\rho},
\end{align*}
which completes the proof.
\end{proof}

We are now easy to see that Theorem \ref{mainthm} is a direct result of Theorem \ref{maingeneralthm} and Proposition \ref{prop_tkmomcond}.

\begin{proof}[Proof of Theorem \ref{mainthm}]
It follows from Proposition \ref{prop_tkmomcond} that $\{t_k\}_{k\in\bN_0}$ given by \eqref{ourthetak} satisfies Momentum-Condition if either $\omega\in(0,1)$, $a\in\bR_+$  or $\omega=1$, $a\in\left(0,\frac{1}{2}\right)$ holds. Then Theorem \ref{mainthm} follows from Theorem \ref{maingeneralthm} immediately.
\end{proof}

From Theorem \ref{maingeneralthm}, we see that the convergence rate depends on the order of $t_{k-1}$. To close this section, we present a proposition showing that a higher order setting of $\{t_k\}_{k\in\bN_0}$ by $t_k:=ak^{\omega}+b$ for $\omega>1$ does not satisfy the second inequality of \eqref{neq_momcondtk1} in Momentum-Condition.

\begin{proposition}\label{prop_tkpowalpha}
Let $t_k:=ak^\omega+b$, $k\in\bN_0$, where $a\neq0$ and $\omega,b\in\bR$. If $\omega>1$, then
$$
\lim_{k\to\infty}t_{k-1}^2-t_k(t_k-1)=-\infty.
$$
\end{proposition}
\begin{proof}
By the definition of $t_k$, we have that
\begin{align}
\notag &t_{k-1}^2-t_k(t_k-1)\\
\notag =&\left[a(k-1)^{\omega}+b\right]^2-(ak^{\omega}+b)^2+ak^{\omega}+b\\
\label{eq1_tkpowalpha}=&-a^2\left[k^{2\omega}-(k-1)^{2\omega}\right]-2ab\left[k^{\omega}-(k-1)^{\omega}\right]+ak^{\omega}+b
\end{align}
It follows from \eqref{lem_kpowomega} that there exists $K\in\bN_+$ such that
$$
k^{2\omega}-(k-1)^{2\omega}\geq\omega k^{2\omega-1},\ \ \mbox{for all}\ \ k>K.
$$
This together with \eqref{eq1_tkpowalpha} yields that for $k>K$,
$$
t_{k-1}^2-t_k(t_k-1)\leq-a^2\omega k^{2\omega-1}-2ab\left[k^{\omega}-(k-1)^{\omega}\right]+ak^{\omega}+b,
$$
which implies that $\lim_{k\to\infty}t_{k-1}^2-t_k(t_k-1)=-\infty$ since the first term on the right-hand side of the above inequality is the highest-order term with respect to $k$ and with a negative coefficient.
\end{proof}

\section{AFBA for smoothed hinge loss L1-SVM model}\label{sec_AFBAforSVM}
Support vector machine (SVM) is one of the most important methods for classification problems. As the data scale in real-world problems grows rapidly, it is crucial to propose efficient algorithms for solving SVM. In this section, we shall use AFBA with the proposed momentum scheme to solve the smoothed hinge loss $\ell_1$-SVM model. The smoothing of the hinge loss function in $\ell_1$-SVM leads to an optimization model of the form \eqref{model1}, so that the resulting model can be solved efficiently by AFBA while preserving the predictive accuracy.

We begin by introducing the original hinge loss $\ell_1$-SVM model. Let $\{(x^{(i)},y^{(i)}):i\in\bN_m\}\subset\bRn\times\{-1,1\}$ be a given training data set, where $\bN_m:=\{1,2,\ldots,m\}$. Let $\mK:\bRn\times\bRn\to\bR$ be a positive semi-definite kernel function, $K:=[\mK(x^{(i)},x^{(j)})]_{i,j=1}^m$ be the kernel matrix associated with $\mK$. The hinge loss function and $\ell_1$ norm are defined by $h_1(t):=\max(1-t,0)$ for $t\in\bR$, and $\|x\|_1:=\sum_{i=1}^{n}|x_i|$ for $x\in\bRn$, respectively. Then the SVM model with $\ell_1$ regularization ($\ell_1$-SVM) is given by
\begin{equation*}
\min_{\alpha\in\bR^{m},b\in\bR}\left\{\sum_{i=1}^{m}\mL(x^{(i)},y^{(i)},\alpha,b)+\lambda\|\alpha\|_1\right\},
\end{equation*}
where $\alpha\in\bR^{m}$ is a vector consisting of the linear combination coefficients, $b\in\bR$ is a bias term, $\lambda\in\bR_+$ is the regularization parameter, and the function $\mL$ in the fidelity term is defined by
\begin{equation*}
\mL(x^{(i)},y^{(i)},\alpha,b):=h_1\left(y^{(i)}\left(\sum_{j=1}^m\alpha_{j}\mK(x^{(j)},x^{(i)})+b\right)\right).
\end{equation*}
By letting
$$
w:=\left[\begin{array}{c}
\alpha\\
b
\end{array}\right]\in\bR^{m+1},\ \
{\bf1}_{m}:=[1,1,\ldots,1]^\top\in\bR^{m},
$$
$$
\tilde{K}:=\left[
\begin{array}{cc}
K & {\bf1}_m
\end{array}\right],\ \  Y:=\diag\{y^{(1)},\ldots,y^{(m)}\},\ \ B:=Y\tilde{K},
$$
\begin{equation}\label{tildeI}
\tilde{I}:=\left[\begin{array}{cc}
I_{m} &{\bf0}_{m}\\
{\bf0}_{m}^\top &0
\end{array}\right]
\end{equation}
and $h(u):=\sum_{i=1}^{m}h_1(u_i)$ for $u\in\bRm$, the $\ell_1$-SVM model can be rewritten by
\begin{equation}\label{L1-MKSVM2}
\min_{w\in\bR^{m+1}}\{h(Bw)+\lambda\|\tilde{I}w\|_1\}.
\end{equation}
The difficulty to develop an efficient algorithm for solving model \eqref{L1-MKSVM2} is that both the function $h$ and the $\ell_1$-norm in the model are nondifferentiable. To address this issue, we smooth the hinge loss function as follows:
\begin{equation}\label{tildeh1}
\tilde{h}_1(t):=\begin{cases}
(1-t)^2, &t<1,\\
0, &t\geq1.
\end{cases}
\end{equation}
It was mentioned in \cite{zhang2001text} that the squared hinge loss function defined by
\eqref{tildeh1} is a better loss function for capturing the heavy tailed distribution, which is more appropriate for classification problems. Now the smoothed hinge loss $\ell_1$-SVM (SHL-$\ell_1$-SVM) model we aim to solve is given by
\begin{equation}\label{SHL-L1-SVM1}
\min_{w\in\bR^{m+1}}\{\tilde{h}(Bw)+\lambda\|\tilde{I}w\|_1\},
\end{equation}
where $\tilde{h}(u):=\sum_{i=1}^{m}\tilde{h}_1(u_i)$ for $u\in\bRm$. Model \eqref{SHL-L1-SVM1} is an instance of Model \eqref{model1} with
\begin{equation}\label{fgforSHLL1SVM}
f=\tilde{h}\circ B\ \ \mbox{and}\ \ g=\lambda\|\cdot\|_1\circ\tilde{I}.
\end{equation}
Before employing AFBA to solve model \eqref{SHL-L1-SVM1}, we verify that function $f$ in \eqref{fgforSHLL1SVM} is convex and differentiable with a Lipschitz continuous gradient. For this purpose, we first show the convexity of $\tilde{h}_1$ and the Lipschitz continuity of its derivative.

\begin{lemma}\label{lemma_tildeh1}
Suppose that $\tilde{h}_1:\bR\to\bR$ is defined by \eqref{tildeh1}. Then $\tilde{h}_1$ is convex and differentiable with a 2-Lipschitz continuous derivative.
\end{lemma}
\begin{proof}
To prove the convexity of $\tilde{h}_1$, it suffices to show that its derivative is monotonically increasing (see Exercise 14 in Chapter 5 of \cite{rudin1976principles}). It follows from the definition of $\tilde{h}_1$ that
$$
\tilde{h}_1^{'}(t)=\begin{cases}
2(t-1), &t<1,\\
0, &t\geq1,
\end{cases}
$$
which is monotonically increasing. We next show that $\tilde{h}_1^{'}$ is Lipschitz continuous. For any $t_{1},t_{2}\in\bR$, without loss of generality, we assume that $t_{1}\leq t_{2}$. Then we have
$$
\tilde{h}_1^{'}(t_1)-\tilde{h}_1^{'}(t_2)=\begin{cases}
2(t_{1}-t_{2}), &t_{1}\leq t_{2}<1,\\
2(t_{1}-1), &t_{1}<1\leq t_{2},\\
0, & 1\leq t_{1}\leq t_{2},
\end{cases}
$$
which implies that
$$
|\tilde{h}_1^{'}(t_1)-\tilde{h}_1^{'}(t_2)|\leq 2|t_1-t_2|,
$$
that is, $\tilde{h}_1^{'}$ is 2-Lipschitz continuous.
\end{proof}

\begin{proposition}\label{LipschitzConstant}
Suppose that $f:\bR^{m+1}\to\bR$ is defined in \eqref{fgforSHLL1SVM}. Then $f$ is convex and differentiable with a $2 \|B\|_2^2$-Lipschitz continuous gradient.
\end{proposition}
\begin{proof}
By Theorem 5.7 of \cite{rockafellar1970convex}, to prove the convexity of $f$, it suffices to show that $\tilde{h}$ is convex. It follows from Lemma \ref{lemma_tildeh1} that $\tilde{h}_1$ is convex and differentiable with a 2-Lipschitz continuous derivative. Using the definition of $\tilde{h}$ and the convexity of $\tilde{h}_1$, for any $u,v\in\bRm$ and $\alpha\in(0,1)$, we have that
\begin{align*}
\tilde{h}((1-\alpha)u+\alpha v)&=\sum_{i=1}^{m}\tilde{h}_1((1-\alpha)u_i+\alpha v_i)\\
&\leq\sum_{i=1}^{m}\left[(1-\alpha)\tilde{h}_1(u_i)+\alpha \tilde{h}_1(v_i)\right]\\
&=(1-\alpha)\tilde{h}(u)+\alpha\tilde{h}(v),
\end{align*}
which implies that $\tilde{h}$ is convex, and hence $f$ is convex.

We next prove that $\nabla f$ is $2\|B\|_2^2$-Lipschitz continuous. By using the chain rule and the Lipschitz continuity of $\tilde{h}_1^{'}$, for any $u,v\in\bR^{m+1}$, we have that
\begin{align*}
\|\nabla f(u)-\nabla f(v)\|&=\|B^\top\nabla\tilde{h}(Bu)-B^\top\nabla\tilde{h}(Bv)\|\\
&\leq\|B^\top\|_2\left[\sum_{i=1}^{m}\left(\tilde{h}_1^{'}((Bu)_i)-\tilde{h}_1^{'}((Bv)_i)\right)^2\right]^{\frac{1}{2}}\\
&\leq 2\|B^\top\|_2\left[\sum_{i=1}^{m}\left((Bu)_i-(Bv)_i\right)^2\right]^{\frac{1}{2}}\\
&=2\|B\|_2\|Bu-Bv\|\\
&\leq 2\|B\|_2^2\|u-v\|,
\end{align*}
which completes the proof.
\end{proof}

We now give the closed form of $\prox_{\mu\|\cdot\|_1\circ\tilde{I}}$.

\begin{proposition}
Suppose that $\tilde{I}$ is defined by \eqref{tildeI} and $\mu\in\bR_+$. Then for $w\in\bR^{m+1}$,
$$
\prox_{\mu\|\cdot\|_1\circ\tilde{I}}(w)=\left[\prox_{\mu|\cdot|}(w_1),\prox_{\mu|\cdot|}(w_2),\cdots,\prox_{\mu|\cdot|}(w_{m}),w_{m+1}\right]^\top.
$$
\end{proposition}
\begin{proof}
It follows from the definition of proximity operator that
\begin{align*}
\prox_{\mu\|\cdot\|_1\circ\tilde{I}}(w)&=\argmin_{v\in\bR^{m+1}}\left\{\frac{1}{2}\|v-w\|_2^2+\mu\|\tilde{I}v\|_1\right\}\\
&=\argmin_{v\in\bR^{m+1}}\left\{\sum_{i=1}^{m}\left(\frac{1}{2}(v_i-w_i)^2+\mu|v_i|\right)+\frac{1}{2}(v_{m+1}-w_{m+1})^2\right\}\\
&=\left[\prox_{\mu|\cdot|}(w_1),\prox_{\mu|\cdot|}(w_2),\cdots,\prox_{\mu|\cdot|}(w_{m}),w_{m+1}\right]^\top,
\end{align*}
which proves the desired result.
\end{proof}

The proximity operator of $\mu|\cdot|$ is given in \cite{micchelli2011proximity} by
$$
\prox_{\mu|\cdot|}(t)=\max(|t|-\mu,0)\cdot\sign(t),\ \ \mbox{for}\ \ t\in\bR.
$$
Now AFBA for solving model \eqref{SHL-L1-SVM1} can be given by
\begin{equation}\label{alg2}
\begin{cases}
v^k=w^k+\theta_k(w^k-w^{k-1}),\\
w^{k+1}=\prox_{\beta\lambda\|\cdot\|_1\circ\tilde{I}}(v^k-\beta B^\top\nabla\tilde{h}(Bv^k)),
\end{cases}
\end{equation}
where $\beta\in\left(0,\frac{1}{2\|B\|_2^2}\right]$, and $w^0, w^1\in\bR^{m+1}$ are given initial vectors.

\section{Numerical experiments}\label{sec_numerical}
In this section, we present the performance of the proposed momentum scheme by comparing it with momentum schemes \eqref{Nesterovthetak} and \eqref{CDthetak}. The three competing momentum schemes are used in conjunction with AFBA to solve the SHL-$\ell_1$-SVM model. We call setting \eqref{CDthetak} the Chambolle-Dossal (CD) momentum scheme, and setting \eqref{ourthetak} the generalized Nesterov (GN) momentum scheme.

In the numerical experiments, we compare the algorithms on two public datasets from LIBSVM \cite{chang2011libsvm}. The first dataset for the comparison is called ``MNIST01", which comes from MNIST handwritting digit database. We only use the digits ``0" and ``1" for classification, leading to a database with 12665 samples in the training set and 2115 samples in the test set, and each sample has 400 features. The second dataset is ``Splice" with 3175 samples and each sample has 60 features. We use 1000 samples to train the model and consider the rest 2175 samples as the test data. All the numerical experiments are implemented on a personal computer with a 2.90 GHz Intel Core i7 processor, a 16GB DDR4 memory and a NVIDIA GeForce GTX 1660 SUPER GPU.

The Gaussian kernel function $\mK(x,y)=e^{-\gamma\|x-y\|^{2}}$ for $x,y
\in\bRn$ was used in the SHL-$\ell_1$-SVM model. Since $\nabla f$ is $2\|B\|_2^2$-Lipschitz continuous (Proposition \ref{LipschitzConstant}), we set the algorithmic parameter $\beta$ to $\frac{1}{2\|B\|_2^2}$. As for the Gaussian kernel parameter $\gamma$ and the regularization parameter $\lambda$, we fine-tune their optimal values according to the performance of test accuracy. This work focuses on the efficiency comparison of the competing momentum schemes. We thus employ AFBA with the three momentum schemes to solve the same SHL-$\ell_1$-SVM model with the fine-tuned $\gamma$ and $\lambda$. We present in Table \ref{Coeff} the optimal values of $\gamma$ and $\lambda$ for the two datasets.

%Size: (TrainSet, TestSet)
\renewcommand\arraystretch{1.5}
\begin{table}[h]
\small
\centering
\caption{Optimal values of parameters $\gamma$ and $\lambda$ for the two datasets.}
\label{Coeff}
\begin{tabular}{|c|c|c|}
 \hline Dataset & MNIST01 & Splice\\
\hline Size (training, test) &(12665, 2115) & (1000, 2175)\\
\hline ($\gamma$, $\lambda$) & ($2^{-5}$, $2^{0}$) & ($2^{-5}$, $2^{-7}$)\\
\hline
\end{tabular}
\end{table}

The numerical results consist of two parts. In the first part, we will compare the performance of AFBA with the GN momentum scheme (AFBA-GN) with different $\omega$. In the second part, we show the comparison of FISTA, AFBA with the CD momentum scheme (AFBA-CD) and the proposed AFBA-GN. Three figure-of-merits will be used in the comparison, including the normalized objective function value (NOFV), the training accuracy and the test accuracy. We define the NOFV by
$$
\text{NOFV}(w^{k}):=\frac{F(w^{k})-F_{ref}}{F(w^{0})-F_{ref}},
$$
where $F$ is the objective function, $F_{ref}$ denotes the reference objective function value. We set $F_{ref}$ to the objective function value at $w^{100000}$ obtained by performing 100,000 iterations of FISTA. Note that the competing algorithms have almost the same computational cost in each iteration, hence we evaluate the performance with respect to the iteration number throughout the numerical section.

\begin{figure}[htbp]
\centering
\subfigure[]{
\includegraphics[width=0.47\linewidth]{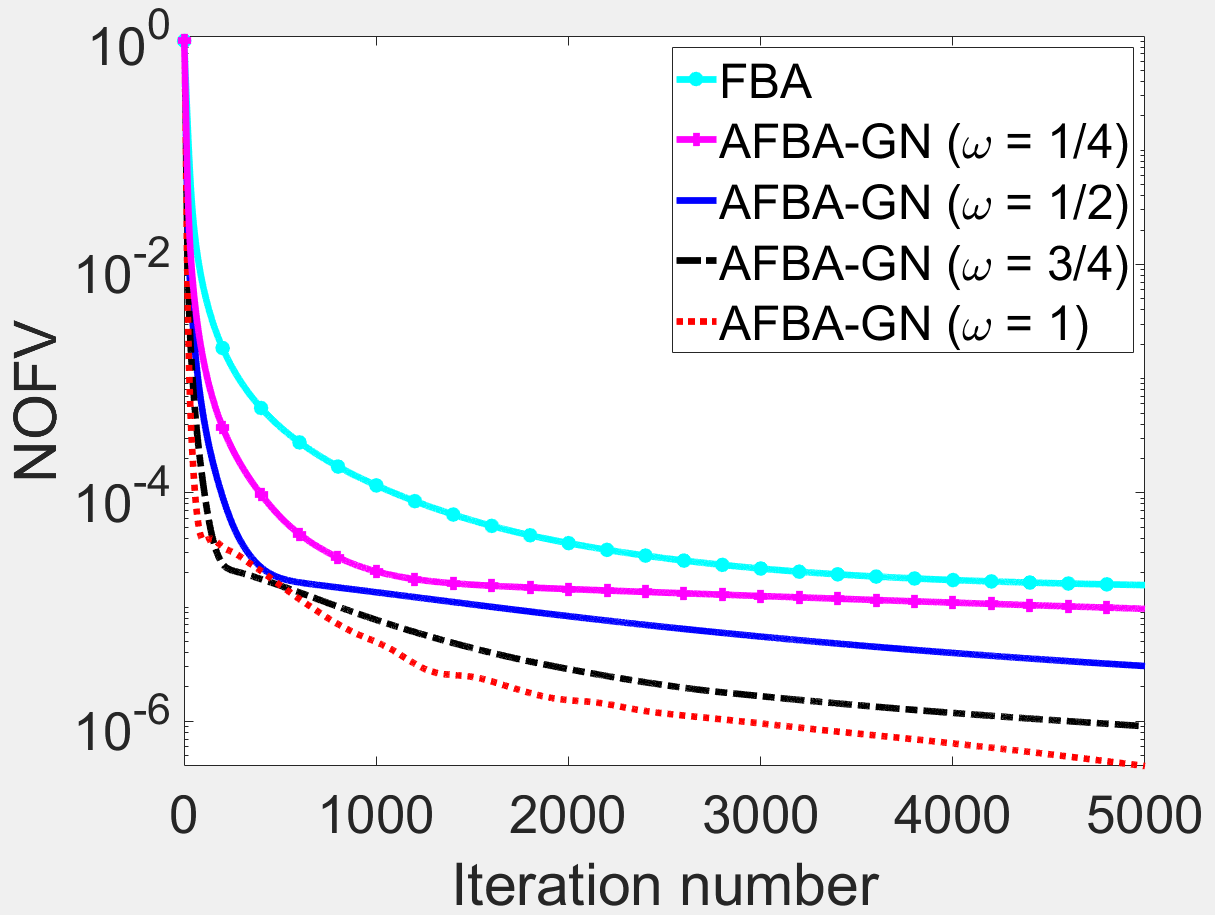}}\\\vspace{-0.5em}
\subfigure[]{
\includegraphics[width=0.47\linewidth]{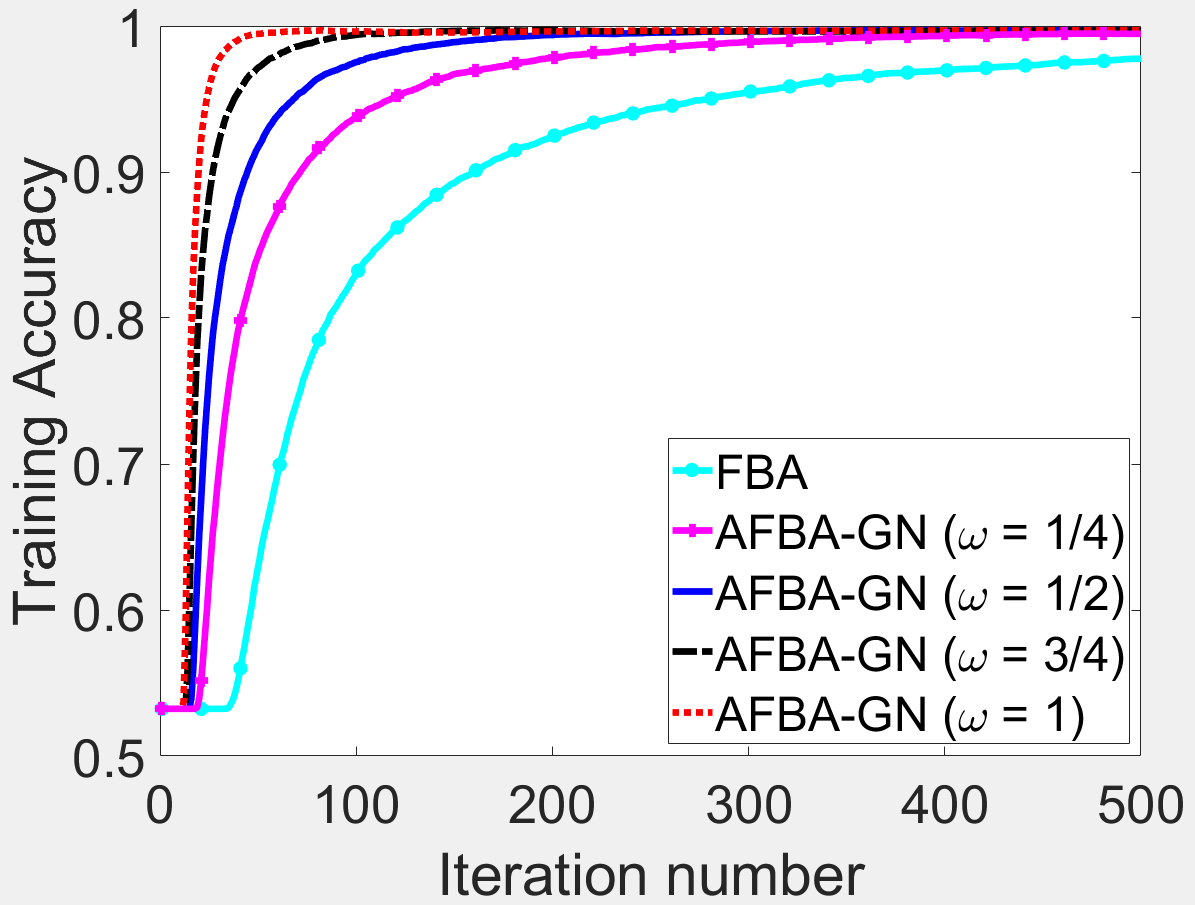}}
\subfigure[]{
\includegraphics[width=0.47\linewidth]{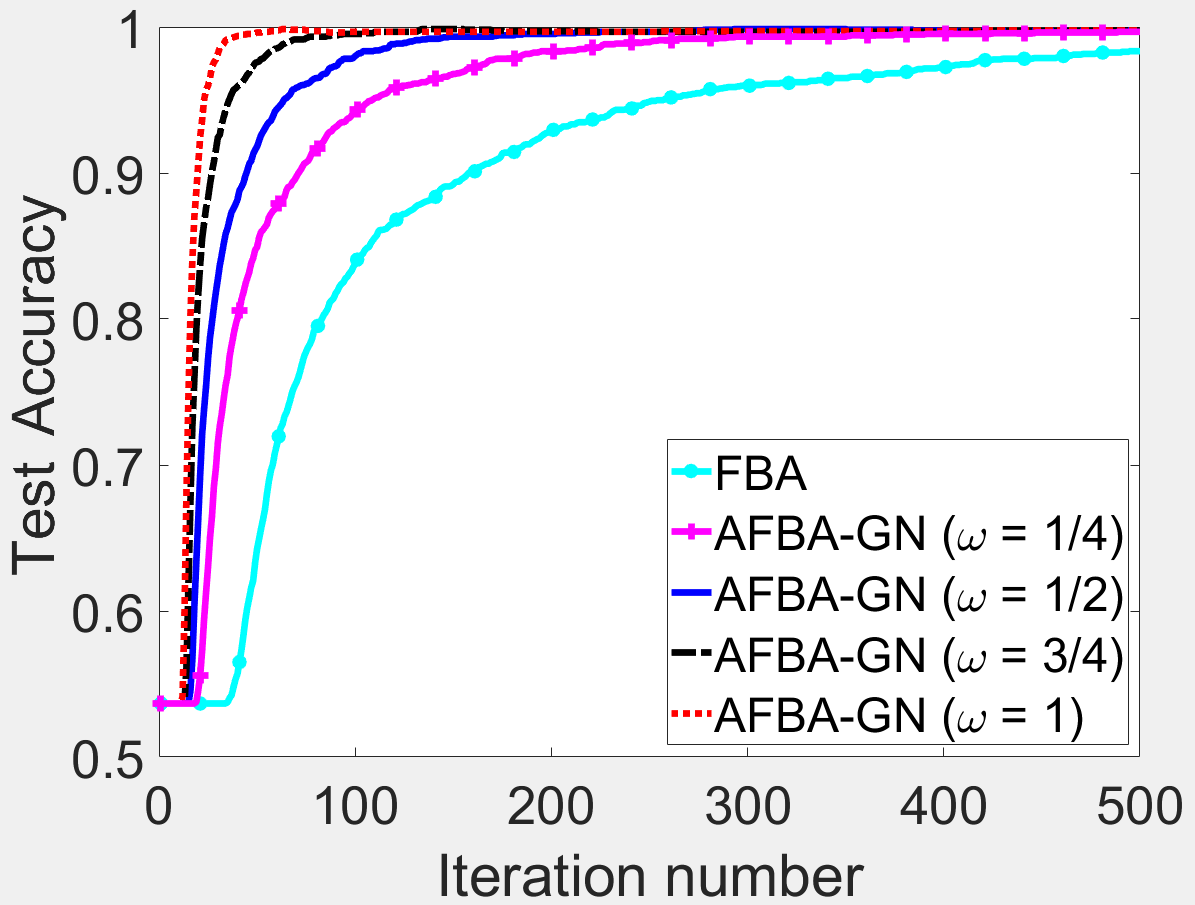}}\vspace{-1em}
\caption{Comparison of FBA and AFBA-GN with $\omega=\frac{1}{4},\frac{1}{2},\frac{3}{4},1$ using ``MNIST01": (a) normalized objective function value versus iteration number; (b) training accuracy versus iteration number; (c) test accuracy versus iteration number.}
\label{fig1}
\end{figure}

In the first experiment, we evaluate the performance of FBA and AFBA-GN with $\omega=\frac{1}{4},\frac{1}{2},\frac{3}{4},1$ in terms of NOFV, training accuracy and test accuracy for the classification of dataset ``MNIST01". In AFBA-GN, we set $a$ and $b$ to $\frac{1}{2.01}$ and 1, respectively, for all cases of $\omega$. From Figure \ref{fig1} (a), (b) and (c), we see that AFBA-GN converges more rapidly than FBA in terms of all the three figure-of-merits. Moreover, larger $\omega$ ($t_{k-1}$ of higher order) gives faster convergence, which is consistent with the convergence rate results in Theorem \ref{mainthm}.

The second experiment presents the behavior of FISTA, AFBA-CD and AFBA-GN with $\omega=1$ and two different sets of $a$, $b$. For parameter $\alpha\in(3,+\infty)$ in AFBA-CD, we empirically found that smaller $\alpha$ gives faster convergence. Therefore, we set $\alpha=3.01$. According to the observation from Figure \ref{fig1}, we always set $\omega=1$ for AFBA-GN. The two sets of parameters $a$, $b$ for comparison are given by $a=\frac{1}{4}$, $b=0$ and $a=\frac{1}{2.01}$, $b=5$, respectively. We recall that the CD momentum scheme is a special case of the GN momentum scheme with $a=\frac{1}{\alpha-1}$, $b=1$, and we can see that $a=\frac{1}{2.01}$ in the second set is consistent with the value of $\alpha$ in AFBA-CD. We empirically found that for the three competing momentum accelerated algorithms with almost the same convergence rate, faster convergence of NOFV may not lead to faster convergence of training and test accuracies (see Figure \ref{fig2}). As a result, the first set of parameters $a$, $b$ was determined based on the best performance of NOFV, while the second set was determined based on the best performance of training and test accuracies.

\begin{figure}[htbp]
\centering
\subfigure[]{
\includegraphics[width=0.5\linewidth]{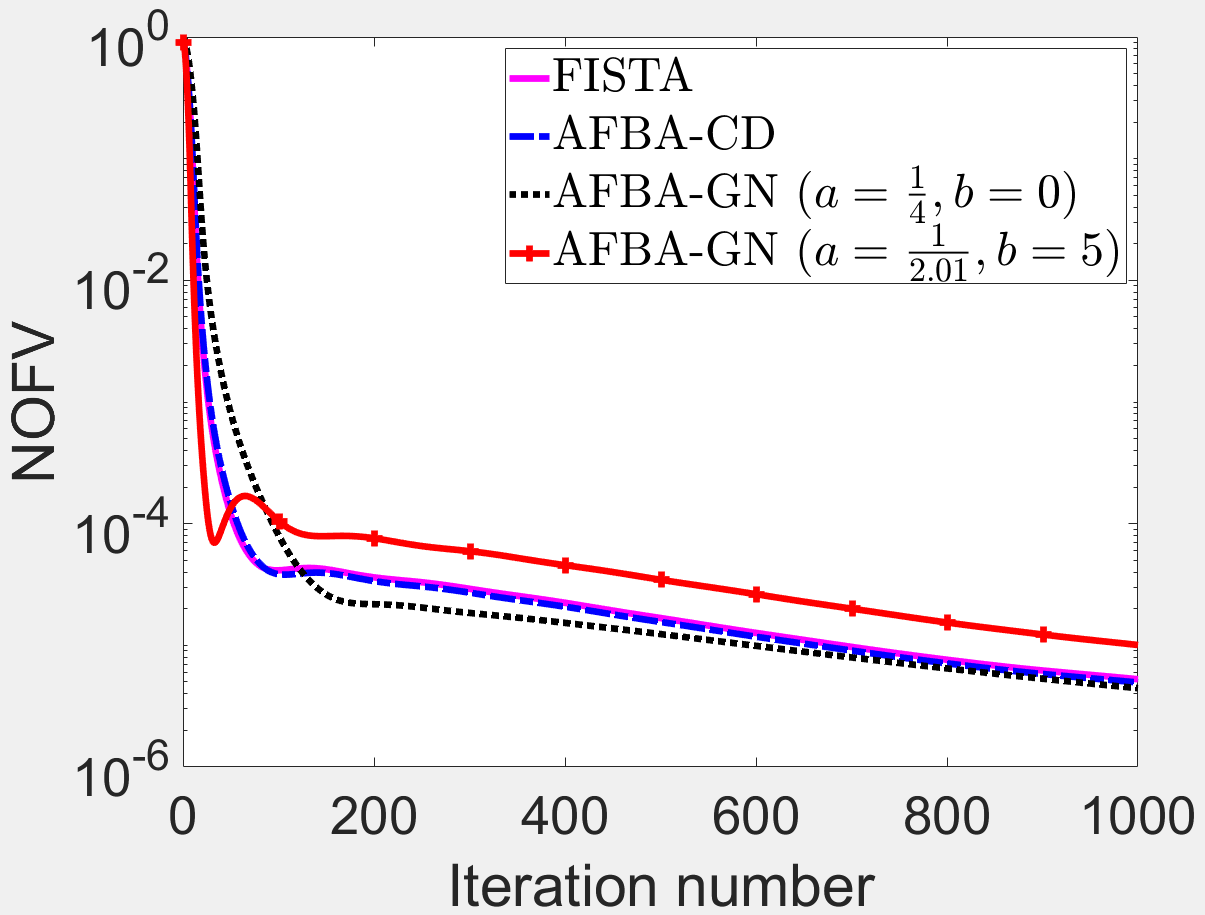}}\\\vspace{-0.5em}
\subfigure[]{
\includegraphics[width=0.48\linewidth]{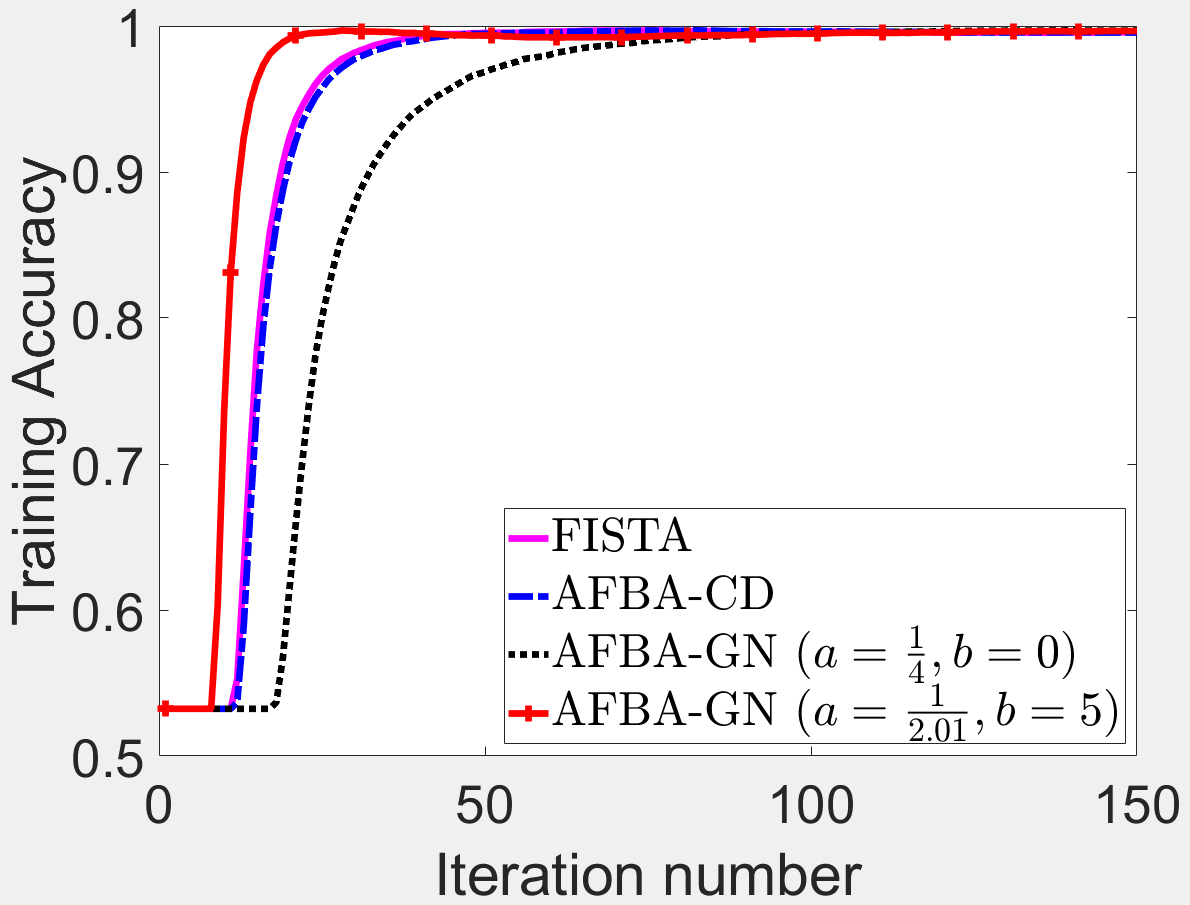}}
\subfigure[]{
\includegraphics[width=0.48\linewidth]{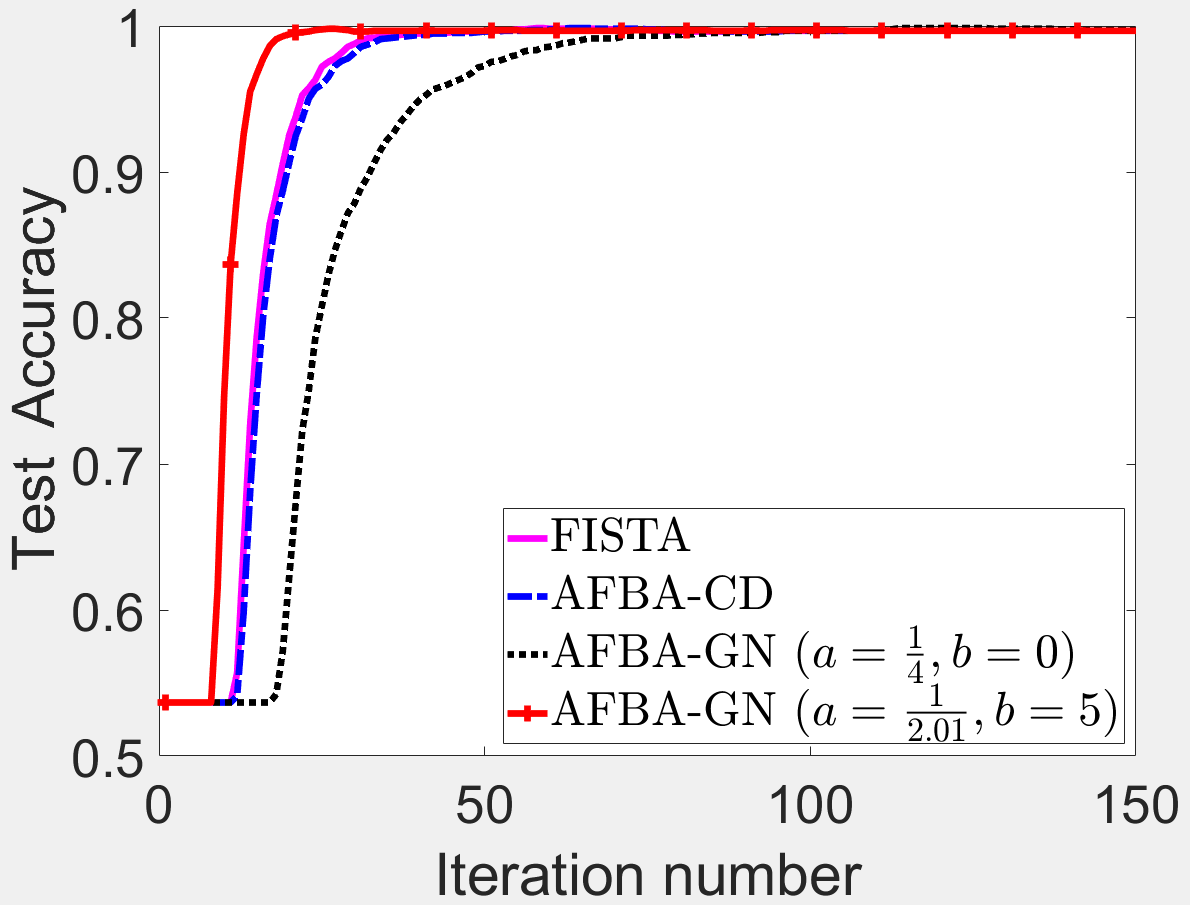}}\vspace{-1em}
\caption{Comparison of FISTA, AFBA-CD, AFBA-GN $(a=\frac{1}{4},b=0)$ and AFBA-GN ($a=\frac{1}{2.01}, b=5$) using ``MNIST01": (a) normalized objective function value versus iteration number; (b) training accuracy versus iteration number; (c) test accuracy versus iteration number.}
\label{fig2}
\end{figure}

As shown in Figure \ref{fig2}, we observe that the plots of FISTA and AFBA-CD almost coincide in terms of all the three figure-of-merits. Figure \ref{fig2} (a) shows that the NOFV plot of AFBA-GN $(a=\frac{1}{2.01},b=5)$ converges the fastest at the early iterations, but is then followed by a mild oscillation. After that, this plot goes a little higher than that of the other algorithms. AFBA-GN $(a=\frac{1}{4},b=0)$ converges the slowest at the early iterations, but achieves the lowest NOFV in subsequent iterations. In fact, we can see that all the plots of the competing algorithms in Figure \ref{fig2} (a) have almost the same convergence speed in the later iterations. The performances of the competing algorithms in terms of training accuracy and test accuracy are shown in Figure \ref{fig2} (b) and (c), respectively. Among these algorithms, AFBA-GN $(a=\frac{1}{2.01},b=5)$ achieves both high training and test accuracies faster than the other algorithms.

In machine learning problems, a model with high training accuracy is prone to overfitting to the training data. As a result, we often care more about test accuracy when a high enough training accuracy is attained. To better evaluate the behaviors of the competing algorithms in terms of test accuracy, we list in Table \ref{AccIterNum} the number of iterations required to achieve various desired levels of test accuracy. We remark that when the number of iterations required is larger than 100,000, it will be marked by `-' in the table. This phenomenon appears in obtaining a 99.9\% test accuracy by FBA due to its slowness. From this table, we are able to see that AFBA-GN $(a=\frac{1}{2.01},b=5)$ can achieve desired levels of test accuracy at about half of the iterations, comparing to FISTA and AFBA-CD.

\begin{figure}[htbp]
\centering
\subfigure[]{
\includegraphics[width=0.47\linewidth]{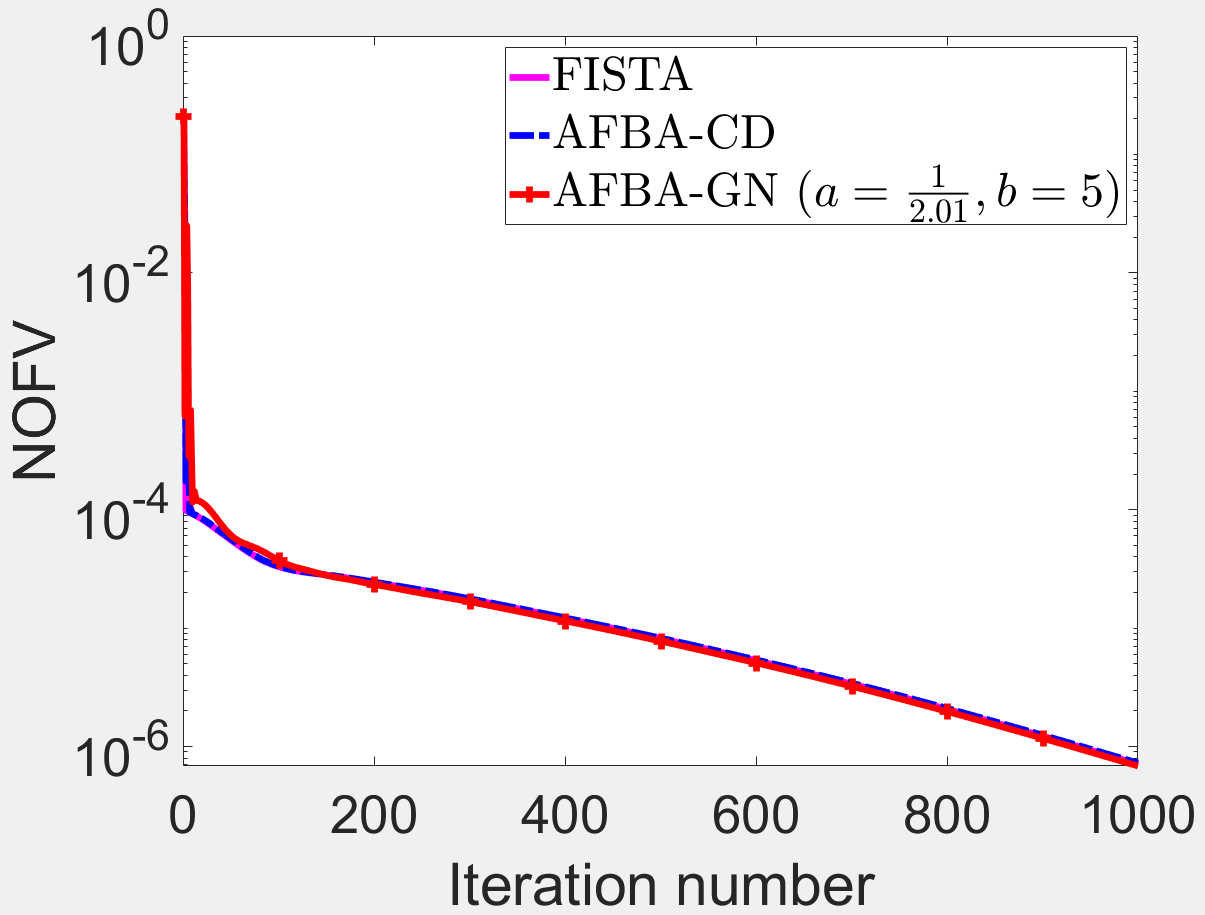}}\\\vspace{-0.5em}
\subfigure[]{
\includegraphics[width=0.47\linewidth]{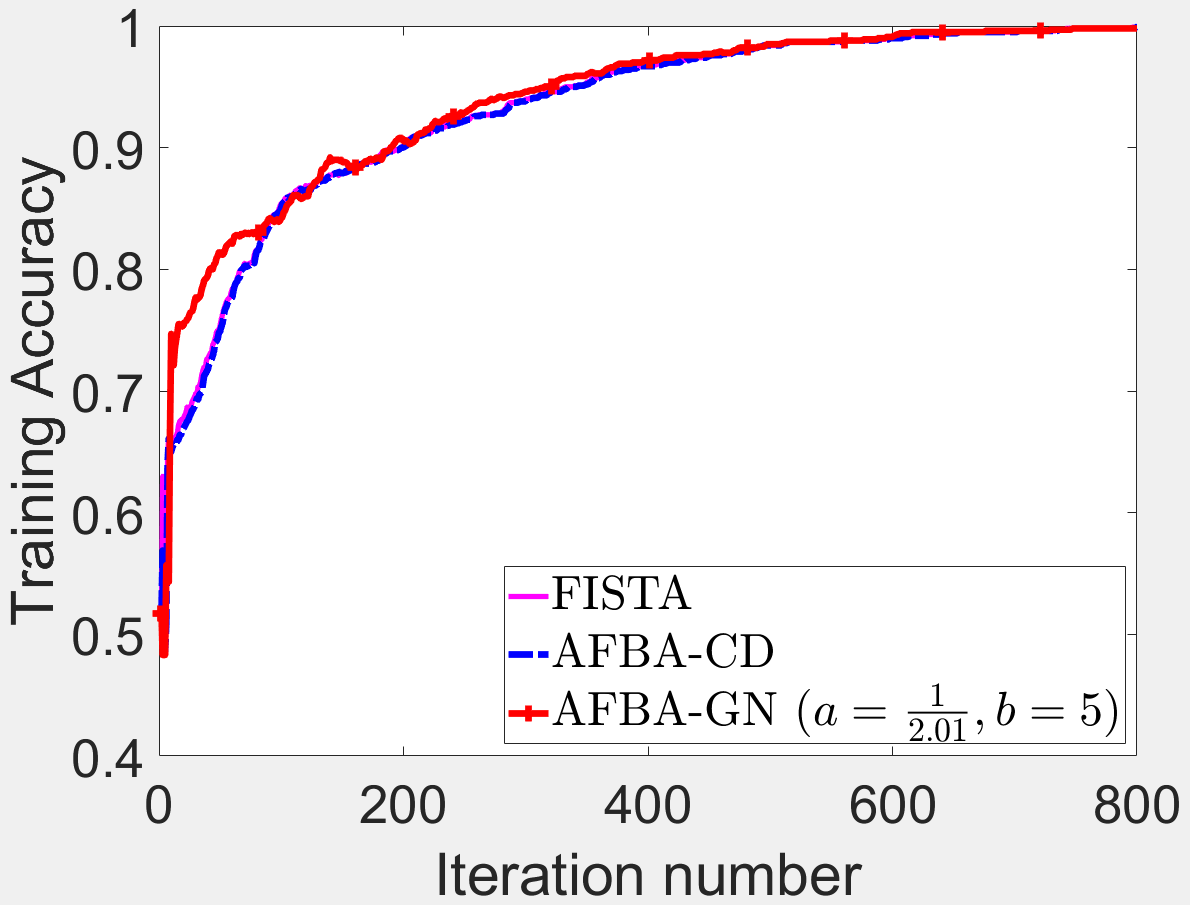}}
\subfigure[]{
\includegraphics[width=0.47\linewidth]{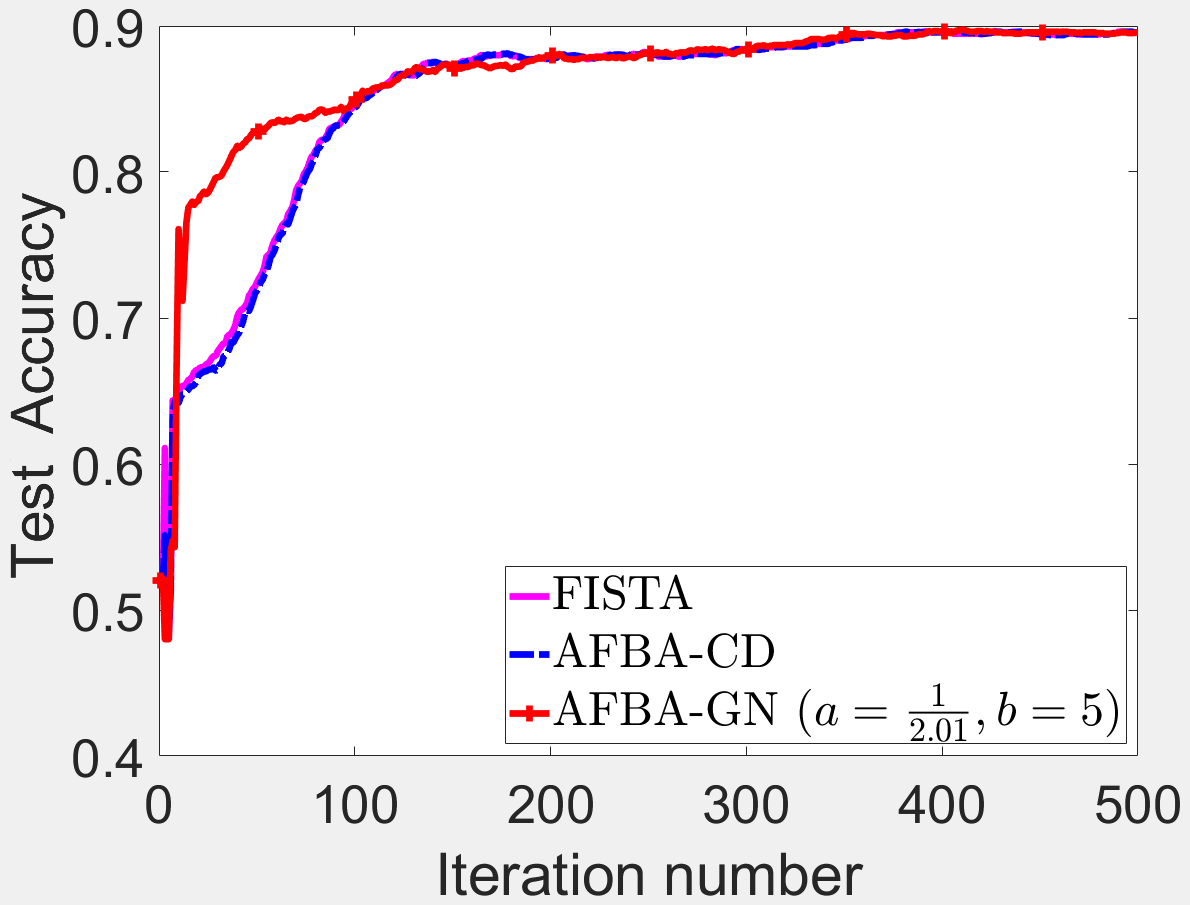}}\vspace{-1em}
\caption{Comparison of FISTA, AFBA-CD and AFBA-GN ($a=\frac{1}{2.01}, b=5$) using ``Splice": (a) normalized objective function value versus iteration number; (b) training accuracy versus iteration number; (c) test accuracy versus iteration number.}
\label{fig3}
\end{figure}

\renewcommand\arraystretch{1.3}
\newcommand{\tabincell}[2]{\begin{tabular}{@{}#1@{}}#2\end{tabular}}
\begin{table}[h]
\footnotesize
\centering
\caption{The number of iterations required to achieve the desired levels of test accuracy.}
\label{AccIterNum}
\begin{tabular}{|c|c|c|c|c|c|c|c|}
\hline
\diagbox[width=2.6cm,height=0.8cm]{Algorithm}{Accuracy} & 90\% & 95\% & 97\% & 99\% & 99.5\% & 99.7\% & 99.9\%\\
\hline
FBA & 158 & 258 & 385 & 661 & 1100 & 1795 & -\\
\hline
FISTA & 19 & 22 & 25 & 31 & 42 & 51 & 1259\\
\hline
AFBA-CD & 20 & 23 & 27 & 34 & 45 & 57 & 1265\\
\hline
\tabincell{c}{AFBA-GN \\\footnotesize{$(a=\frac{1}{2.01},b=5)$}} & {\bf13} & {\bf14} & {\bf16} & {\bf18} & {\bf21} & {\bf24} & {\bf620}\\
\hline
\end{tabular}
\end{table}

Finally, we show in Figure \ref{fig3} the performances of FISTA, AFBA-CD and AFBA-GN $(a=\frac{1}{2.01},b=5)$ for the classification of the other dataset ``Splice". In this experiment, the three competing algorithms have almost the same convergence speed in terms of NOFV. For the training and test accuracies, AFBA-GN $(a=\frac{1}{2.01},b=5)$ still outperforms the other two algorithms.

\section{Conclusion}\label{sec_conclusion}
This paper proposes a generalized Nesterov (GN) momentum scheme for the accelerated forward-backward algorithm (AFBA). We prove the convergence of the iterative sequence generated by AFBA with the GN momentum scheme (AFBA-GN). Moreover, we show that AFBA-GN has an $o\left(\frac{1}{k^{2\omega}}\right)$ convergence rate in terms of the function value, and an $o\left(\frac{1}{k^\omega}\right)$ convergence rate in terms of the distance between consecutive iterates, where $\omega\in(0,1]$ is a power parameter introduced in the GN momentum scheme. The generality of the proposed momentum scheme provides a wider class of momentum algorithms with various convergence rates depending on $\omega$. The specific class of GN momentum scheme with $\omega=1$ is still more general than the existing Chambolle-Dossal (CD) momentum scheme, which may lead to superior performance in various scenarios. In the numerical experiments on support vector machine, the performances of the GN scheme in terms of various figure-of-merits can be optimized via fine-tuning the momentum parameters. The results demonstrate that the GN scheme outperforms the Nesterov and the CD momentum schemes. This shows that AFBA with the GN momentum scheme has great potential for classification problems.

\bibliographystyle{siamplain}
\bibliography{bibfile}

%\newpage

%\sectice on*{Acknowledgments}
%The author thanks the anonymous referees for their useful suggestions.

\end{document}